\documentclass[12pt,reqno, fullpage]{amsart}
\usepackage{enumitem}

\usepackage[all]{xy}
\usepackage{hyperref}
\usepackage{cleveref}
\usepackage{amsmath,amssymb,amscd,comment}
\usepackage{mathtools}

\usepackage[bottom]{footmisc}
\usepackage{lipsum}
\usepackage{setspace,caption}
\usepackage[verbose]{placeins}

\numberwithin{equation}{section}

\usepackage{tikz-cd}

\usetikzlibrary{cd}

\def\varequals#1{\@Arrow@type\@Linecap\@Linecap{#1}{1}}

\usepackage{hyperref}
\hypersetup{
	colorlinks,
	citecolor=red,
	filecolor=black,
	linkcolor=blue,
	urlcolor=black
}

\usepackage[T1]{fontenc}
\theoremstyle{plain}
\newtheorem{thm}{Theorem}[section]
\newtheorem{cor}[thm]{Corollary}
\newtheorem{lmm}[thm]{Lemma}
\newtheorem{prp}[thm]{Proposition}

\newtheorem{que*}{Question}

\newtheorem{property}[thm]{Property}

\newtheorem*{thm*}{Theorem}

\newtheorem{mainthm}{Theorem}

\theoremstyle{definition}
\newtheorem{rem}[thm]{Remark} 
\newtheorem{defn}[thm]{Definition}
\usepackage{amsmath,stackengine}
\usepackage{scalerel,stackengine}
\stackMath
\newcommand\reallywidehat[1]{%
\savestack{\tmpbox}{\stretchto{%
\scaleto{%
    \scalerel*[\widthof{\ensuremath{#1}}]{\kern-.6pt\bigwedge\kern-.6pt}%
    {\rule[-\textheight/2]{1ex}{\textheight}}
  }{\textheight}%
}{0.5ex}}%
\stackon[1pt]{#1}{\tmpbox}%
}

\newcommand\groupequation[2][21pt]{%
  \setbox0=\hbox{$\displaystyle#2$}%
  \stackengine{0pt}{\copy0}{%
    \makebox[\linewidth]{\hfill$\left.\rule{0pt}{\ht0}\right\}$\kern#1}}
    {O}{c}{F}{T}{L}
}

\def \PP{\mathbb{P}}
\def \ZZ{\mathbb{Z}}

\def \ev{\mathrm{ev}}

\begin{document}
\title[Gromov-Witten invariants in family and Quantum cohomology]{Gromov-Witten invariants in family and Quantum cohomology}

\author[I. Biswas]{Indranil Biswas}

\address{Mathematics Department, Shiv Nadar University, NH91, Tehsil Dadri, Greater Noida, Uttar Pradesh 
201314, India}

\email{indranil.biswas@snu.edu.in, indranil29@gmail.com}

\author[N. Das]{Nilkantha Das}

\address{Stat-Math Unit, Indian Statistical Institute, 203 B.T. Road, Kolkata 700 108, India}

\email{dasnilkantha17@gmail.com}

\author[J. Oh]{Jeongseok Oh}
\address{Department of Mathematical Sciences and Research Institute of Mathematics, 
Seoul National University,
Seoul 08826,
Korea.}
\email{jeongseok@snu.ac.kr}

\author[A. Paul]{Anantadulal Paul}
\address{Survey No. 151, Shivakote, Hesaraghatta, Uttarahalli Hobli, Bengaluru, 560089, India.}
\email{anantadulal.paul@icts.res.in}

\maketitle

\begin{abstract}
A moduli space of stable maps to the fibers of a fiber bundle is constructed.
The new moduli space is a family version of the classical moduli space
of stable maps to a non-singular complex projective variety. The virtual cycle for this moduli space is also constructed, and an analogue of Gromov-Witten invariants is defined. As an application, we recover the formula for the number
of rational degree $d$ curves in $\mathbb{P}^3$, whose image lies in a plane in $\mathbb{P}^3$ (known as
planar curves in $\mathbb{P}^3$), intersecting $r$ general lines while passing through given
$s$ general points, where $r + 2s = 3d + 2$, firstly proved by R. Mukherjee, R.
Kumar Singh and the fourth named author.
\end{abstract}

\section{Introduction}
\subsection*{Enumerative problem via Gromov-Witten theory}

Let $X$ be an irreducible smooth complex projective variety. The Gromov-Witten (GW for short) invariants of $X$ 
naively count curves in $X$ of given degree and genus. These counts are rational numbers and are defined by the 
intersection theory on $\overline{\mathcal{M}}_{g,n}(X, \beta)$, the moduli spaces of stable maps representing 
a class $\beta \,\in\, H_2(X,\, \ZZ)$ (cf. \cite{FuPa,K.M}). Here $g$ denotes the arithmetic genus of curves. The 
moduli space $\overline{\mathcal{M}}_{g,n}(X,\beta)$ may neither be smooth nor irreducible, and hence it does not 
possess a fundamental class (of homogeneous degree). Behrend, Behrend-Fantechi and Li-Tian (see 
\cite{Behrend,B.F, L.Ti2}), have however defined the virtual fundamental class
\begin{equation*}
 \left[ \overline{\mathcal{M}}_{g, n}(X, \beta) \right] ^{\mathrm{vir}} \,\in\,
H_{2 \mathrm{vd}}(\overline{\mathcal{M}}_{g,n}(X,\beta),\, \mathbb{Q}), 
\end{equation*}
where $\mathrm{vd} \,:=\, c_1(T_X)\cdot \beta + (\dim X -3)(1-g) + n$.
Given cycles $\gamma_1,\, \ldots, \,\gamma_n \,\in\, H^{\ast}(X,\, \ZZ)$, the associated primary Gromov-Witten invariant is defined as follows:
\begin{align*}
\langle \gamma_1,\, \ldots,\, \gamma_n\rangle_{g,n, \beta}\ =\ \int_{\left[ \overline{\mathcal{M}}_{g,n}(X, \beta) \right]^{\mathrm{vir}} } \ev_1^{\ast}(\gamma_1) \cdots \ev_n^{\ast}(\gamma_n).
\end{align*}
This theory is developed through both algebro-geometric approach (cf. \cite{Behrend,B.F,B.M,FuPa,K.M})
and symplecto-geometric approach (cf. \cite{Ionel-symp-sum,L.Ti2,McSa,RT}).

In some appropriate cases, GW-invariants coincide with the classical count of curves, that is the GW-invariants
are enumerative (for instance, see \cite{RVelliptic}). Over the last few decades, GW-theory
attained great success on enumeration of curves on projective varieties. Particularly in low genera, this
theory plays a crucial role amongst other curve counting theories available in literature. For example, when
$X\,=\, \mathbb{P}^2$, the WDVV equation (which is equivalent to the associativity of the quantum cohomology
of $X$) produces the famous recursive formula of Kontsevich-Manin counting the rational curves on $\mathbb{P}^2$
of a given degree. More generally the following interesting question may be asked:

\begin{que*} \label{genus-q-count-in-plane}
How many degree $d$ curves of genus $g$ are there in $\PP^2$ passing through $3d+g-1$
points in general position?
\end{que*}

It essentially amounts to computing the Gromov-Witten invariant of $\PP^2$ of degree $d$ and genus $g$ (cf. \cite{RV}). 

\subsection*{Answer in classical algebraic geometry} Consider the following classical question: 

\begin{que*}\label{nodal_question}
What is $N_d^{\delta}$, the number of degree $d$ curves in $\mathbb{P}^2$ that have $\delta$ distinct nodes and pass through ${\textstyle\frac{d(d+3)}{2}} - \delta$ generic points?
\end{que*}

Question \ref{genus-q-count-in-plane} and Question \ref{nodal_question} both are linked through the degree-genus formula $g
\,=\,\frac{(d-1)(d-2)}{2}$, or equivalently $3d+g-1\,=\,\frac{d(d+3)}{2}$, for a genus $g$ plane curve. Question \ref{nodal_question} was studied extensively more than a hundred years ago from several perspectives by several mathematicians. A complete solution to this question was given by Caporaso-Harris \cite{CH}. 

\subsection*{Generalization to a family}

In a series of papers \cite{KP3,KP4,KP2}, Kleiman and Piene studied a natural generalization of the above question --- that of counting nodal curves in a moving family of surfaces.
Considering a family of surfaces $\pi\,:\, E\, 
\longrightarrow \,B$ over a $\mathbb{C}$-scheme $B$, they counted $\delta$-nodal curves lying in fibers,
for $\delta \,\leq\, 8$, and conjectured for the general nodal case. In \cite[Theorem A]{TL}, Laarakker made some 
partial progress on that conjecture. As a special case, one can consider the planar curves in 
$\mathbb{P}^3$. A curve in $\mathbb{P}^3$ is said to be \textit{planar} if it lies inside some $\mathbb{P}^2 
\,\subset\, \mathbb{P}^3$. Note that $\mathbb{G}\,:=\,G(3,4)$, the Grassmannian of $3$-planes in 
$\mathbb{C}^4$, parametrizes the space of all planes in $\mathbb{P}^3$. The projectivization of the 
tautological vector bundle $E\,:=\, \mathbb{P}(\gamma) \,\xrightarrow{\,\,\,\pi\,\,\,}\, B\,:=\,\mathbb{G}$ is then the 
$\mathbb{P}^2$-bundle that encodes all the information of all the planes in $\mathbb{P}^3$. A planar curve in 
$\mathbb{P}^3$ is a curve in $E_b\,:= \,\pi^{-1}(b)$ for some point $b$ of $B$. As an application of his theorem, 
Laarakker, \cite[Theorem B]{TL}, computed the number of $\delta$-nodal degree $d$ planar curves in 
$\mathbb{P}^3$ meeting appropriate number of generic lines, for all $\delta$. This can be viewed as a fiber 
bundle version of the plane curve counting. Furthermore, this question has been studied for more degenerate 
singularities in \cite{Das-Mukherjee}.

As for how Question \ref{genus-q-count-in-plane} arises, in terms of Gromov-Witten invariants, from Question 
\ref{nodal_question}, it is natural to ask the analogue of \Cref{genus-q-count-in-plane} in the planar version 
setting. Using stable map approach Mukherjee-Paul-Singh \cite{BMS} found a solution of the following 
question.

\begin{que*}\label{Q3}
What is $N_d(r,s)$, the number of rational degree $d$ planar curves in $\mathbb{P}^3$ that intersect $r$
general lines and also pass through $s$ points in $\mathbb{P}^3$ satisfying the condition $r+2s \,=\, 3d+2$?
\end{que*}

We will describe their precise answer to the question in the next section after introducing moduli spaces. Their
strategy was to use Kock-Vainsencher's approach (cf. \cite{KV0}) after assuming the smoothness of genus $0$
moduli spaces. In this paper this question will be answered in a different way:

\begin{itemize}
\begin{samepage}
\item after developing a concrete curve counting theory, and
\item using WDVV equation.
\end{samepage}
\end{itemize}

There is another motivation to develop the curve counting theory.
Building up on the results of Mukherjee-Paul-Singh \cite{BMS}, Mukherjee-Singh, in \cite{Rahul_Rit}, obtained
a formula for the characteristic number of planar rational degree $d$ curves in $\mathbb{P}^3$ having
a cusp. This naturally requires to study the moduli space of stable maps in a moving family of targets, as
well as its Gromov-Witten analogue.

\subsection*{Gromov-Witten invariants for a family} 

Let us now take a closer look at the stable map version of planar curves. We fix the notation $\pi\,:\,E\,
\longrightarrow\, B$ for 
the family of planes $\mathbb{P}(\gamma) \,\longrightarrow\, \mathbb{G}$. A \textit{planar $n$-pointed stable
map} is a stable 
map $\mu\,:\, (C,\,p_1,\, \ldots ,\,p_n) \,\longrightarrow\, E$ such that $\pi \circ \mu$ is a constant
map. In other words, if $\pi 
\circ \mu\,= \,c$, then $\mu$ lands inside $E_c\,:=\, \pi^{-1}(c) \,\cong\, \mathbb{P}^2$.

Let us now try to implement the above idea on a general fiber bundle. Let $B$ be a non-singular variety, $F$
a non-singular projective variety, and $E\, \xrightarrow{\,\,\,\pi\,\,\,}\, B$ a $F$-fiber bundle. Given $g,\,n
\,\geq\, 0$ and $\beta \,\in\, H_2(F,\,\mathbb{Z})$, we want to define a moduli space parametrizing
stable maps $\mu\,:\, (C,\,p_1,\, \ldots ,\,p_n)\, \longrightarrow\, E$ such that $\pi \circ \mu$ is
a constant map and $\mu_*\left( [C]\right)\,=\,\beta$ after identifying $F$ with the fiber into which
$\mu$ lands. Then a new moduli space can be defined using this data. Let us denote the moduli
space by $\overline{\mathcal{M}}(g,n,\beta)$ (which depends on $g$, $n$ and $\beta$). This moduli space, if
it exists, is equipped with a map $\overline{\mathcal{M}}(g,n,\beta) \,\longrightarrow\, B$ whose fibers
are isomorphic to the moduli space $\overline{\mathcal{M}}_{g,n}(F, \beta)$. It is very desirable
for $\overline{\mathcal{M}}(g,n,\beta)$ to be a $\overline{\mathcal{M}}_{g,n}(F, \beta)$-fiber bundle. For
such a moduli space to exist, we need to put some restrictions on the fiber bundle $E\, \longrightarrow\, B$.
Consider the following subgroup of $\mathrm{Aut}(F)$:
\begin{equation}\label{the structure group_1}
G \,:=\, \{ \phi \,\in\, \mathrm{Aut}(F) \,\big\vert\,\, \phi_*\, \in\, \text{Aut}(H_2(F,\, \mathbb{Z}))\,\,
\text{ is } \,\, {\rm Id}_{H_2(F,~ \mathbb{Z})}\}.
\end{equation}	
If we assume $E$ to be a $F$-fiber bundle over $B$ with the above structure group $G$, the existence of 
$\overline{\mathcal{M}}(g,n,\beta)$ with desired properties can be ensured. The technical assumption on the 
structure group is discussed in more detail in \Cref{section family version of moduli spaces}.

\begin{mainthm}
Let $B$ be a non-singular variety, $F$ a non-singular projective variety, and $E \,\xrightarrow{\,\,\,\pi\,\,\,}\,
B$ a $F$-fiber bundle with structure group $G$ as defined in \cref{the structure group_1}. Given $g,\,n \,
\geq \,0$ and an effective homology class $\beta \,\in\, H_2(F,\,\mathbb{Z})$, the moduli space
$\overline{\mathcal{M}}(g,n,\beta) \,\longrightarrow\, B$ discussed above exists and it is
a $\overline{\mathcal{M}}_{g,n}(F, \beta)$-fiber bundle. Hence, in particular, it is a proper Deligne-Mumford
stack if $B$ is proper. Furthermore, it comes equipped with a natural virtual fundamental class
$[\overline{\mathcal{M}}(g,n,\beta)]^{\textnormal{vir}}$ of virtual dimension 
\[
c(T_F).\beta+(\dim F-3)(1-g)+n+\dim B,
\]
which is the same as the virtual dimension of $\overline{\mathcal{M}}_{g,n}(F, \beta)$ plus the dimension
of $B$.
\end{mainthm}

It maybe worthwhile to mention that $\overline{\mathcal{M}}_{g,n}(F, \beta)$ is a singular DM stack. Our approach of constructing $\overline{\mathcal{M}}(g,n,\beta)$ is a singular version of smooth fiber bundle constructed by using the automorphisms of (singular) DM stacks.

There are natural evaluation maps $\mathrm{ev}_i\,:\, \overline{\mathcal{M}}(g,n,\beta)\,\longrightarrow\, E$ given 
by evaluation on the $i^{\textnormal{th}}$ marked point so that we can define
the numerical invariants as usual. Given 
cohomology cycles on $E$, we can pull them back via the maps $\mathrm{ev}_i$ and integrate over 
$[\overline{\mathcal{M}}(g,n,\beta)]^{\textnormal{vir}}$. These are analogues of Gromov-Witten invariants.

Let us now return to the planar version set-up and formulate \Cref{Q3} in terms of the above numerical invariants. 
Recall that $\mathbb{P}(\gamma)\,\longrightarrow\, \mathbb{G}$ is the fiber bundle in this setting. The fiber bundle 
$\mathbb{P}(\gamma)$, being a sub-bundle of $\mathbb{G}\times \mathbb{P}^3$, induces a natural projection map 
$\mathbb{P}(\gamma) \,\longrightarrow\, \mathbb{P}^3$. Let $a\,\in\, H^2(\mathbb{G})$ and $H\,\in\, 
H^2(\mathbb{P}(\gamma))$ denote respectively the hyperplane classes of $\mathbb{G}$ and the pullback of the hyperplane 
classes of $\mathbb{P}^3$; the lines and points are considered to be the Poincar\'e duals of $H^2$ and $H^3$, 
respectively. We now define
\begin{align}\label{ndrst}
N_d(r,s,\theta)\,:=\,\int_{[\overline{\mathcal{M}}(0,r+s,d[\textnormal{line}])]^{\textnormal{vir}}}
\mathrm{ev}_1^*(a^{\theta}H^2)\cdot \left(\prod_{i=2}^r \ev_i^*(H^2)\right) \cdot
\left( \prod_{j=r+1}^{r+s} \ev_j^*(H^3)\right).
\end{align}
The parameter $\theta$ encodes the codimension of the cycle $a^{\theta}$.
Observe that $N_d(r,s,0)$ gives the virtual count for \Cref{Q3}, whereas $N_d(r,s,3)$ gives that of \Cref{genus-q-count-in-plane} with $g=0$. In fact, we will show that $N_d(r,s,\theta)$ is enumerative in \Cref{enumerative significance}.
 
Note that 
\begin{align*}
		N_d(r,s,\theta) &=  \begin{cases} 0 & \mbox{if} ~~r + 2s + \theta \neq 3d+2, \\ 
		0 & \mbox{if }  ~~s>3, \mbox{ or } \theta >3.\\ 
\end{cases} 
		\end{align*}
The first vanishing follows from the fact that the virtual dimension is not the same as the codimension of the cycle. Since $a^{\theta}=0$ for $\theta \ > \ 3$, $N_d(r,s,\theta)$ vanishes in this case.  Geometrically, for any planar curve passing through $s$ many generic points,
we have that the plane on which the curve lies also passes through $s$ many generic points. But there is no plane in $\mathbb{P}^3$ passing through 
more than three generic points. So, $N_{d}(r,s,\theta)$ vanishes as well.

Then we reprove the following recursion formula, first proved in \cite[Theorem 3.3]{BMS}, which answers 
\Cref{Q3} by setting $N_d(r,s)\,=\,N_d(r,s,0)$.

\begin{mainthm}\label{thm_B}
For $d \,\geq\, 2$ and $r\,\geq\, 3$, the following recursion relation holds:
\begin{align*}
&N_d(r,s,\theta)\,=\,\,  2dN_d(r-1,s,\theta+1)-2d^2N_d(r-2,s,\theta+2) \\
&+ \sum_{\substack{d_1,d_2 > 0\\ d_1+d_2=d}}\sum_{\substack{r_1+r_2=r-1\\ s_1+s_2=s \\ \theta_1+\theta_2=\theta+3}} \left( d_1^2d_2^2\binom{r-3}{r_1-1}- d_1^3d_2{r-3\choose r_1}\right) \binom{s}{s_1} N_{d_1}(r_1,s_1,\theta_1)N_{d_2}(r_2,s_2,\theta_2)\nonumber
\end{align*}
with the convention that ${a \choose b}\,=\,0$ if $b\,<\,0$ or $b\,>\,a$ while the initial condition is
given by 
\begin{align*}
		N_d(r,s,\theta) & \,= \begin{cases}
		1 &  \mbox{if} ~~(d,r,s,\theta) =(1,0,2,1),(1,2,1,1), (1,1,1,2),(1,2,0,3),(2,2,3,0)\\ 
		0 & \mbox{\textnormal{otherwise with $d = 1$ and $r \geq 3$}}\\
		0 & \mbox{\textnormal{otherwise with $d\geq 1$ and $r\leq 2$}}.\end{cases} 
\end{align*}
\end{mainthm}
\smallskip

\noindent Using the recursion in Theorem \ref{thm_B}, one can check that the number $N_3(11,0,0)$ of rational planar cubics passing through 11 lines in $\mathbb{P}^3$ is $12960$ which is known in classical algebraic geometry (cf. \cite[Appendix A]{TL} and \cite[Section 4]{BMS}).

\subsection*{Plan of the paper}

In Section \ref{section family version of moduli spaces}, we study the moduli space of stable maps to a fiber 
bundle $\pi\,:\,E\,\longrightarrow\, B$ whose typical fiber is $F$. As in the usual case of moduli problem of stable maps, it is 
shown that the new moduli problem is a fine moduli space as a Deligne-Mumford stack and it is coarsely represented 
by a projective variety (see \Cref{moduli_existence}). We construct it as a fiber bundle over $B$ with fiber 
$\overline{\mathcal{M}}_{g,n}(F, \beta)$.

In Section \ref{family version GW invariants} we construct a virtual cycle for the moduli space to define 
invariants using integrations over it. With a mild abuse of nomenclature, we will continue to call these 
intersection numbers as fiberwise GW-invariants, as the geometry suggests, these invariants are family versions
of the usual GW-invariants. To distinguish, we may use the name ``ordinary GW-invariants'' for the invariants 
that are defined using the Kontsevich moduli space $\overline{\mathcal{M}}_{g,n}(F, \beta)$.

In an analogy with the ordinary GW-theory, these newly defined fiberwise GW-invariants also satisfy various properties 
including the WDVV equation in $g\,=\,0$. We prove this in Theorem \ref{WDVV theorem} in Section \ref{quantum 
cohomology}. Finally, in \Cref{rational planar curve section}, we apply the WDVV equation to count the 
rational planar curves on $\mathbb{P}^3$ of degree $d$, providing a proof of \Cref{thm_B}.

\section{Family version of moduli spaces}\label{section family version of moduli spaces}

\subsection*{Target space}

Let $F$ be a non-singular projective variety and $\pi\,:\,E\,\longrightarrow\, B$ a $F$-fiber bundle over a
non-singular base $B$. Given integers $g,\,n \,\geq\,0$ and $\beta \,\in\, H_2(F,\,\mathbb{Z})$, we want to
construct a moduli space parametrizing stable maps $\mu\,:\, (C,\,p_1,\, \ldots ,\,p_n)\,\longrightarrow\,
E$ defined over $B$ whose fibers are isomorphic to the moduli space $\overline{\mathcal{M}}_{g,n}(F, \beta)$. To
construct it generally, we encounter the following problem:  \\
Let $(U\,\subset\, B,\, \phi\,:\,\pi^{-1}U\,\xrightarrow{\,\sim\,}\, U\times F)$ and $(V\,\subset\, B,\,
\psi\,:\,\pi^{-1}V\,\xrightarrow{\,\sim\,}\, V\times F)$ be two trivializations of the fiber bundle such
that $U \cap V$ is non-empty. Then the induced map $$\phi \circ \psi^{-1}\,:\, U \cap V \,
\longrightarrow\, \mathrm{Aut} (F)$$ yields an isomorphism $\overline{\mathcal{M}}_{g,n}(F, \beta)
\,\longrightarrow\, \overline{\mathcal{M}}_{g,n}(F, \left(\phi \circ \psi^{-1}(b)\right)_*(\beta))$ for
$b\,\in\, U \cap V$. Thus the transition maps of the moduli space may depend on the choices of trivializations. To solve the problem we consider the subgroup of $\mathrm{Aut}(F)$:
\begin{equation}\label{thep}
G\, :=\, \{ \phi \,\in \,\mathrm{Aut}(F) \,\big\vert\,\,
\phi_*\, \in\, \text{Aut}(H_2(F,\, \mathbb{Z}))\,\,
\text{ is } \,\, {\rm Id}_{H_2(F,~ \mathbb{Z})}\}
\end{equation}	
introduced in the introduction (cf. \cref{the structure group_1}); we assume $E$ to be a $F$-fiber bundle over $B$ with structure group $G$. 

\subsection*{Moduli space} For any scheme $S$ over $\mathbb{C}$, a \textit{family of maps} over $S$ from
$n$-pointed genus $g$ curves to $E$, and compatible with $\pi$, consists of the data
\begin{equation}\label{df}
\left(\widetilde{\pi}\,:\, \mathcal{C} \,\rightarrow\, S, \,\{ p_i\}_{1 \leq i \leq n},\,
\mu\,:\, \mathcal{C } \,\rightarrow\, E, \, \rho\,:\, S \,\rightarrow\, B \right)
\end{equation}
such that
\begin{enumerate}[label=\textnormal{(\roman*)}]
\item $\widetilde{\pi}\,:\, \mathcal{C}\,\longrightarrow \,S$ is a
family of $n$-pointed genus $g$ prestable curves equipped with $n$ disjoint sections
$\{ p_1,\, \cdots,\, p_n\}$ of $\widetilde{\pi}$, and

\item the diagram
$$
\xymatrix{
 \mathcal{C} \ar[rr]^{\mu} \ar[d]_{\widetilde{\pi}} & & E \ar[d]^{\pi}\\
 S \ar[rr]_{\rho} & & B
}
$$
is commutative.
\end{enumerate}
Two such families of maps $\left(\widetilde{\pi}: \mathcal{C} \rightarrow S, \{ p_i\}, \mu ,  \rho  \right)$
and $\left(\pi^{\prime}: \mathcal{C}^{\prime} \rightarrow S, \{ p_i^{\prime}\}, \mu^{\prime},
\rho^{\prime}\right)$ are isomorphic if there exists an isomorphism $\tau\,:\, \mathcal{C}
\,\longrightarrow\, \mathcal{C}^{\prime}$ for which $\mu^{\prime} \circ \tau \,=\, \mu$,
$\pi^{\prime} \circ \tau \,=\, \widetilde{\pi}$ and $\tau \circ p_i\,=\, p_i^{\prime}$. Notice that these imply
that $\rho\,=\, \rho^{\prime}$. A family of pointed maps $\left(\widetilde{\pi}\,:\, \mathcal{C}
\,\rightarrow \,S,\, \{ p_i\},\, \mu ,\,  \rho\right)$ is stable if the pointed family
$\left(\widetilde{\pi}\,:\, \mathcal{C} \,\rightarrow\, S, \,\{ p_i\},\, \mu\right)$ is
stable (cf. \cite[Definition 1.1]{FuPa}). 

Take $\beta \,\in\, H_2(F,\,\mathbb{Z})$. Denote by $R_2\pi_*\underline{\mathbb{Z}}$ the constant sheaf on
$B$
\begin{align}\label{R_2}
R_2\pi_*\underline{\mathbb{Z}}\ :=\ R^{2\dim F-2}\pi_*\underline{\mathbb{Z}}.
\end{align}
Note that the local system $R^{2\dim F-2}\pi_*\underline{\mathbb{Z}}$ is really a constant sheaf of
abelian groups due to the choice of our structure group $G$ for $\pi$ (see \eqref{thep}).
Let $\widetilde{\beta}$ be the global section of $R_2\pi_*\underline{\mathbb{Z}}$
corresponding to (the Poincar\'e dual of) $\beta$. Consider the family of
pointed maps over $S$ in \eqref{df}. For each $s$ in $S$, let $\mu_s\,:\, \mathcal{C}_s \,\longrightarrow\,
E_{\rho(s)}$ be the induced morphism of $\mu$ between the base changes $\mathcal{C}_s\,:=\,
\widetilde{\pi}^{-1}(s)$ and $E_{\rho(s)}\,:=\, \pi^{-1}(\rho(s))$. 
Then the family in \eqref{df} is said to
\textit{represent the class $\widetilde{\beta}$} if we have $\left(\mu_s\right)_*\left( 
[\mathcal{C}_s]\right)\,=\, \widetilde{\beta}\big\vert_{\rho(s)}$.

The \textit{moduli functor} 
\begin{equation}\label{moduli functor}
\overline{\mathcal{M}}^{\mathrm{Fib}}_{g,n}(E/B,\,\widetilde{\beta})\ :\ Schemes/ 
\mathbb{C} \,\longrightarrow\, Sets
\end{equation}
is defined as follows:
$\overline{\mathcal{M}}^{\mathrm{Fib}}_{g,n}(E/B,\,\widetilde{\beta})(S)$ is
the set of isomorphism classes of families of $n$-pointed genus $g$ stable maps to $E$ over $S$ compatible 
with $\pi$ and representing the class $\widetilde{\beta}$. Note that when $B$ is a point,
then the moduli functor becomes 
the usual moduli functor $\overline{\mathcal{M}}_{g,n}(F, \beta)$ of $n$-pointed genus $g$ stable maps 
to $F$ representing the class $\beta$, as defined in \cite[Definition 1.1]{FuPa}. Thus the moduli functor 
considered in \eqref{moduli functor} is a fiber bundle analogue of the usual one.

\subsection*{Existence}
Given a projective variety $X$, and a class $\beta \,\in\, H_2(X,\,\mathbb{Z})$, the classical moduli
space $\overline{\mathcal{M}}_{g,n}(X, \beta)$ is a fine moduli space as a Deligne-Mumford stack, and it is
coarsely represented by a projective variety $\overline{M}_{g,n}(X,\beta)$. We want to do the same for the
moduli functor $\overline{\mathcal{M}}^{\mathrm{Fib}}_{g,n}(E/B,\widetilde{\beta})$ given in
\eqref{moduli functor}. If $\overline{\mathcal{M}}^{\mathrm{Fib}}_{g,n}(E/B,\widetilde{\beta})$ is
coarsely represented by a scheme, then it is unique and the set of complex points of it is a one-to-one
correspondence with $\overline{\mathcal{M}}^{\mathrm{Fib}}_{g,n}(E/B,\widetilde{\beta})(\mathrm{Spec}
(\mathbb{C}))$, that is, the set of all $n$-pointed, genus $g$ stable maps
$\mu\,:\,(C,\, p_1,\, \ldots ,\,p_n)\,\longrightarrow\, E$ such that $\pi \circ \mu$ is a constant
map, up to isomorphism. Equivalently, this is the set of $n$-pointed, genus $g$ stable maps
$\mu\,:\,(C,\, p_1,\, \ldots ,\,p_n)\,\longrightarrow\, E$ such that $\mu(C)$ lies inside a fiber. We denote
by $\overline{M}^{Fib}_{g,n}(E/B,\widetilde{\beta})$ the coarse moduli space, if it exists. Then we get natural maps 
$$
\overline{\mathcal{M}}^{\mathrm{Fib}}_{g,n}(E/B,\widetilde{\beta})\ \longrightarrow\
\overline{M}^{\mathrm{Fib}}_{g,n}(E/B,\widetilde{\beta})\ \longrightarrow\ B
$$ 
whose fiber over each point $b\,\in\, B$ is $\overline{\mathcal{M}}_{g,n}(E_b, \beta)
\,\longrightarrow\, \overline{M}_{g,n}(E_b, \beta)$. 

Recall that the structure group of the fibration $\pi\,:\,E\,\longrightarrow\, B$ is $G$. So $H_2(\pi^{-1}(b),\,
{\mathbb Z})$ is canonically identified with $H_2(F,\, {\mathbb Z})$ for every $b\, \in\, B$.
Let $$\varpi\, :\, E_G \, \longrightarrow\, B$$ be the principal $G$--bundle over $B$ corresponding to
$\pi$. So the fiber $\varpi^{-1}(b)$ over any $b\, \in\, B$ is the space of all isomorphisms
$F\, \longrightarrow\, \pi^{-1}(b)$ that act on $H_2(F,\, {\mathbb Z})$ as the identity map.
There are natural homomorphisms
\begin{equation}\label{eh}
H_s\,:\, G \,\longrightarrow\, {\rm Aut}(\overline{\mathcal{M}}_{g,n}(F,\, \beta)) \ \ \,\text{ and }\ \
\, H_m\, :\, G \,\longrightarrow\, {\rm Aut}(\overline{M}_{g,n}(F, \,\beta)).
\end{equation}
Now define 
\begin{equation}\label{eh2}
\overline{\mathcal{M}}\, \longrightarrow\, B\ \ \, \text{ and }\ \
\,\overline{M}\, \longrightarrow\, B
\end{equation}
to be the fiber bundles associated to the
principal $G$--bundle $E_G$ for the homomorphisms $H_s$ and $H_m$ respectively in \eqref{eh}. From the
constructions of the fiber bundles in \eqref{eh2} it follows immediately that there is a natural morphism
\begin{equation}\label{eh3}
\overline{\mathcal{M}}\, \longrightarrow\, \overline{\mathcal{M}}^{\mathrm{Fib}}_{g,n}(E/B,\,\widetilde{\beta}).
\end{equation}
It is straightforward to check that the morphism in \eqref{eh3} is an isomorphism. For each $b \,\in\, B$, the
fiber of the bundle $\overline{M} \,\longrightarrow\, B$ over $b$ is $\overline{M}_{g,n}(E_b, \beta)$. Let us denote the fiber bundle $\overline{M}$ by $\overline{M}^{\mathrm{Fib}}_{g,n}(E/B,\widetilde{\beta})$. 

Since $\overline{\mathcal{M}}_{g,n}(F,\, \beta)$ is a proper separated Deligne--Mumford stack, from the above construction
of $\overline{\mathcal{M}}$ it follows immediately that it is also
a Deligne--Mumford stack. The underlying coarse moduli space for $\overline{\mathcal{M}}_{g,n}(F,\, \beta)$
is $\overline{M}_{g,n}(F,\, \beta)$. From this it follows immediately that 
the underlying coarse moduli space for $\overline{\mathcal{M}}$ is $\overline{M}$ and the composition
of maps $\overline{\mathcal{M}}\,\longrightarrow\, \overline{M}\,\longrightarrow\, B$ recovers the one
in \eqref{eh2}. The above discussion may be summarized in the following theorem:

\begin{thm}\label{moduli_existence}
Let $B$ be a non-singular projective variety and $\pi\,:\,E \,\longrightarrow\, B$
an $F$-fiber bundle with structure group $G$ be defined as in \cref{thep} such that $F$ is a non-singular
projective variety. Let $\beta \,\in\, H_2(F,\,\mathbb{Z})$ be an effective homology class, and $\widetilde{\beta}$ be
the corresponding global section of $R_2\pi_* \underline{\mathbb{Z}}$ (\textnormal{cf.} \eqref{R_2}).
Then the following two statements hold:
\begin{enumerate}[label=\textnormal{(\roman*)}]
\item The moduli functor $\overline{\mathcal{M}}^{\mathrm{Fib}}_{g,n}(E/B,\,\widetilde{\beta})$ defined
in \cref{moduli functor} is a proper Deligne-Mumford stack (DM-stack for short).

\item The fiber bundle $\overline{M}^{\mathrm{Fib}}_{g,n}(E/B,\widetilde{\beta})\,\longrightarrow\, B$ is
the coarse moduli space for $\overline{\mathcal{M}}^{\mathrm{Fib}}_{g,n}(E/B,\,\widetilde{\beta})$.
\end{enumerate}
\end{thm}

\smallskip
To simplify the notation, we denote by $\overline{\mathcal{M}}_{g,n}(E/B, \beta)$ (respectively, 
$\overline{M}_{g,n}(E/B,\beta)$) the space 
$\overline{\mathcal{M}}^{\mathrm{Fib}}_{g,n}(E/B,\,\widetilde{\beta})$ (respectively, 
$\overline{M}^{\mathrm{Fib}}_{g,n}(E/B,\,\widetilde{\beta})$) unless there is scope for confusion. The 
reader should not be confused with the usual notation $\overline{\mathcal{M}}_{g,n}(X, \beta)$. Here 
$\beta$ is a homology class of $X$, whereas in our case, it is a class of $F$ and not of $E$.

\subsection*{Smoothness when $g=0$} For each point $b\,\in\, B$, let
\begin{equation}\label{j1}
\overline{\mathcal{M}}_{g,n}^*(E_b,\beta) \,\subseteq\, \overline{\mathcal{M}}_{g,n}(E_b,\beta)
\end{equation}
be the open locus of stable maps with no non-trivial 
automorphism. This yields a fiber bundle $\overline{\mathcal{M}}^*_{g,n}(E/B, \beta) \,\subseteq\, 
\overline{\mathcal{M}}_{g,n}(E/B, \beta)$ over $B$ with fiber $\overline{\mathcal{M}}_{g,n}^*(F,\beta)$. We 
finally aim to find the WDVV equations. Hence we focus more on the case of $g\,=\,0$. Theorem 2 and 3 of 
\cite{FuPa} can immediately be extended to the set-up of fiber bundles.

\begin{thm}\label{properties of coarse moduli space}
Consider the $F$-fiber bundle $\pi\,:\, E \,\longrightarrow\, B$ as in \Cref{moduli_existence}. Let $F$
and $B$ be non-singular complex projective varieties. Furthermore, assume $F$ is a convex variety. Then the following statements hold:
\begin{enumerate}[label=\textnormal{(\roman*)}]
\item $\overline{\mathcal{M}}_{0,n}^{*}(E/B,\beta)$ in \eqref{j1} is a smooth proper separated Deligne-Mumford stack of pure dimension $\dim B + \dim \overline{\mathcal{M}}_{0,n}(F,\beta)$. 

\item The boundary of $\overline{\mathcal{M}}_{0,n}(E/B,\beta)$ is a divisor with normal crossings.

\item The coarse moduli $\overline{M}_{0,n}^{*}(E/B,\beta)$ of $\overline{\mathcal{M}}_{0,n}^{*}(E/B,\beta)$ is a non-singular fine moduli space.
\end{enumerate}
\end{thm}
\smallskip

\subsection*{Natural structures}
We now discuss some natural structures of $\overline{\mathcal{M}}_{g,n}(E/B, \beta)$ which will play an important role later on.

Given a family $\left(\widetilde{\pi}: \mathcal{C} \rightarrow S, \{ p_i\}, \mu: \mathcal{C } \rightarrow E,
\rho: S \rightarrow B \right)$ of fibered stable maps to $E$ representing the class $\widetilde{\beta}$, we
have natural morphisms $\phi_i\,:\, S \,\longrightarrow\, E$ given by $\mu \circ p_i$. Thus it defines a bundle map
\begin{equation}\label{evaluation map}
\mathrm{ev}_i\ :\ \overline{\mathcal{M}}_{g,n}(E/B, \beta) \,\longrightarrow\, E
\end{equation}
which sends each closed point $ [\mu: \left(C,p_1, \ldots,p_n \right) \rightarrow E ]$ to $\mu(p_i)$. These
are called as the \textit{evaluation maps}.

When $2g-2+n\,>\,0$, we have the stabilization map $\mathfrak{m}_{g,n}\,\longrightarrow\,
\overline{\mathcal{M}}_{g,n}$ from the moduli space of prestable curves to that of stable ones. Composing
the forgetful map $\overline{\mathcal{M}}_{g,n}(E/B, \beta)\,\longrightarrow\, \mathfrak{m}_{g,n}$ with
the stabilization map, we obtain a map
\begin{equation}\label{forgeting target}
\overline{\mathcal{M}}_{g,n}(E/B, \beta)\,\longrightarrow\, \overline{\mathcal{M}}_{g,n},
\end{equation}
which, with an abuse of terminology, is called the stabilization map.

When $\beta\,=\,0$, the usual moduli space $\overline{\mathcal{M}}_{g,n}(F,\beta)$ is the same as
$\overline{\mathcal{M}}_{g,n}\times F$. In our set-up, there is a similar description of
$\overline{\mathcal{M}}_{g,n}(E/B, 0)\,\cong\, \overline{\mathcal{M}}_{g,n}\times E$ when $\beta\,=\,0$. 

For each fiber, the natural forgetful map $\pi_{n+1}\,:\, \overline{\mathcal{M}}_{g,n+1}(F, \beta)\,
\longrightarrow\, \overline{\mathcal{M}}_{g,n}(F, \beta)$ that simply forgets the $(n+1)^{\text{th}}$
marked points can be thought of as the universal curve over $\overline{\mathcal{M}}_{g,n}(F, \beta)$. Combining
together we derive a map of bundles 
\begin{equation}\label{forgetful map one mark point}
\pi_{n+1}\,:\, \overline{\mathcal{M}}_{g,n+1}(E/B, \beta)\,\longrightarrow\, \overline{\mathcal{M}}_{g,n}(E/B,\beta).
\end{equation}
This can be thought of as the universal curve over $\overline{\mathcal{M}}_{g,n}(E/B,\beta)$.

When $B\,=\,\mathrm{Spec}(\mathbb{C})$, the moduli space $\overline{\mathcal{M}}_{g,n}(E/B, \beta)$ is the classical moduli space $\overline{\mathcal{M}}_{g,n}(E,\beta)$ which is the same as $\overline{\mathcal{M}}_{g,n}(F,\beta)$. Restricted to each fiber, the maps in \eqref{evaluation map}, \eqref{forgeting target}, \eqref{forgetful map one mark point} are morphisms of DM-stacks.

\section{Family version of Gromov-Witten invariants}\label{family version GW invariants}
\subsection{Virtual fundamental class}

Here, $B$ is a nonsingular variety, but it is no longer assumed to be projective.
As before, $F$ is smooth projective.
With the notation of \cite{Gillet} and \cite{Vistoli_intersection_theory}, the moduli space
$\overline{\mathcal{M}}_{g,n}(F, \beta)$ carries a natural cycle lying in its Chow group, called the virtual
fundamental class. Similar to the usual case, the moduli space $\overline{\mathcal{M}}(E/B)$ may
not be smooth, irreducible or equidimensional. In order to define intersection numbers, we first need
to construct the virtual fundamental class in this setting. 

Let us look at Behrend and Fantechi's construction \cite[p.~83]{B.F} of the virtual class of 
$\overline{\mathcal{M}}_{g,n}(F, \beta)$. Let $\pi\,:\, C \,\longrightarrow\, \overline{\mathcal{M}}_{g, n}(F, \beta)$ 
be the universal curve, and let $f\,:\,C \,\longrightarrow\, F$ be the universal stable map.
Then Behrend and Fantechi showed 
that $\left( R\pi_*f^*T_F\right)^{\vee}$ gives rise to a relative perfect obstruction theory with respect to 
the forgetful map $p\,:\,\overline{\mathcal{M}}_{g,n}(F, \beta)\,\longrightarrow\, \mathfrak{m}_{g,n}$
to the moduli space of prestable curves. Recall that
the cone of $(R\pi_*f^*T_F)^\vee[-1] \,\longrightarrow\, p^*T^*_{\mathfrak{m}_{g,n}}$ is a natural 
absolute perfect obstruction theory. With the help of the obstruction theory, they constructed the virtual 
fundamental class. A similar construction will be done here.

Let $C$ be an algebraic curve over $B$ which is flat over $B$. Let us denote the $B$-space 
$\mathrm{Mor}_B(C,\,E)$, the space of $B$-morphisms from $C$ to $E$, by $M$. Let $f\,:\,C \times_B M \,
\longrightarrow\, E$ be the 
universal morphism, and let $\pi\,:\, C \times_B M\,\longrightarrow\, M$ be the natural projection. By the
functorial property, we get the following homomorphism of cotangent complexes:
$$
f^*T^*_{E/B}\ \longrightarrow\ \mathbb{L}_{C\times_B M/B}\ \longrightarrow\ \mathbb{L}_{C\times_B M/C}.
$$
Since $C$ is flat over $B$, the homomorphism 
$$
\pi^*\mathbb{L}_{M/B}\ \longrightarrow\ \mathbb{L}_{C\times_B M/C}
$$ 
is an isomorphism. Combining these we get an induced homomorphism 
$$
e\,:\, f^*T^*_{E/B}\ \longrightarrow\ \pi^*\mathbb{L}_{M/B}.
$$
Assume that $C$ is an algebraic curve with at worst nodes as singular points. Then $C\times_B M$ has a relative dualizing sheaf $\omega$. By Serre duality,
$$
K^{\bullet}\ :=\ R\pi_*\left( f^* T_{E/B} \otimes \omega\right)[1]\ =\ \left(R\pi_*f^*T_{E/B}\right)^{\vee},
$$
where $(M^{\bullet})^\vee$ is the derived dual of $M^{\bullet}$ in the derived category. Thus we obtain an induced morphism 
\begin{equation}\label{obstruction morphism}
\phi\,:=\, R\pi_*(e\otimes\omega[1])^{\vee}\,:\, K^{\bullet} \,\longrightarrow\,
R\pi_*\pi^*(\mathbb{L}_{M/B}\otimes\omega[1])\,\cong\, R\pi_*\pi^!\mathbb{L}_{M/B}
\,\longrightarrow\, \mathbb{L}_{M/B}.
\end{equation}
In \cite{B.F}, Behrend-Fantechi proved that $\phi$, given by \eqref{obstruction morphism}, is a relative perfect obstruction theory of $M$ over $B$.

Our next goal is to show that $\phi$ defines in fact an absolute perfect obstruction theory.

\begin{thm}
As before, $B$ is a non-singular variety, and $\pi\,:\,E \,\longrightarrow\, B$ is a $F$-fiber bundle
with structure group $G$, where $F$ is smooth projective.
Let $\beta \,\in\, H_2(F,\, \mathbb{Z})$. Then there is a perfect obstruction theory
$$\mathbb{E}_{\overline{\mathcal{M}}_{g,n}(E/B,\beta)}\ \longrightarrow\
\mathbb{L}_{\overline{\mathcal{M}}_{g,n}(E/B,\beta)}$$ with the virtual dimension $\mathrm{vd}$ being
\begin{align}\label{vdim_1}
\mathrm{vd}\, =\, (\dim (F)-3)(1-g) + \int_{\beta} c_1(T_F) +n + \dim B.
\end{align}
\end{thm}

\begin{proof}
Throughout the proof, we denote the moduli space $\overline{\mathcal{M}}_{g,n}(E/B,\beta)$ by
$\overline{\mathcal{M}}(E/B)$, for notational convenience.
Replacing $M$ by $\overline{\mathcal{M}}(E/B)$, which is an open substack of the $B$-mapping space
$\mathrm{Map}_{\mathfrak{m}_{g,n}\times B}(C,E)$ where $C$ denotes the universal curve over
$\mathfrak{m}_{g,n}\times B$, the morphism $\phi$ in \eqref{obstruction morphism} gives
\begin{align}
\label{POB-complex}
\phi\,:\, K^{\bullet}\,:=\, \left(R\varpi_*f^*T_{E/B}\right)^{\vee}\,\longrightarrow\,
\mathbb{L}_{\overline{\mathcal{M}}(E/B)/{\mathfrak{m}_{g,n} \times B}},
\end{align}
where $\varpi$ and $f$ are the maps in the following universal diagram
\begin{align*}
\xymatrix{
C \times_{\mathfrak{m}_{g,n}\times B} \overline{\mathcal{M}}(E/B)  \ar[r]^-{f}\ar[d]^{\varpi} & {E/B}\\
\overline{\mathcal{M}}(E/B).
}
\end{align*}
We will show that $\phi$ in \eqref{POB-complex} gives a relative perfect obstruction theory. For that
it is enough to prove the following statements:
\begin{itemize}
\item $h^{0}( \phi)$ is isomorphism,

\item $h^{-1}(\phi)$ is surjection, and

\item $K^{\bullet}$ has cohomology only in degrees $-1$, $0$.
\end{itemize}
These are all local properties, so we check them locally. Each point in $B$ has a neighbourhood $U$ over which 
\begin{itemize}
\item $\overline{\mathcal{M}}(E/B)$ is a product $\overline{\mathcal{M}}_U\,
:=\,\overline{\mathcal{M}}_{g,n}(F,\beta)\times U$,
\item $E$ becomes a product $F\times U$,
\item $\phi$ restricts to $\phi_U$, the pullback of
$$
\phi_F\,:\, (R\pi_*f^*T_F)^\vee\,\longrightarrow\, \mathbb{L}_{\overline{\mathcal{M}}_{g,n}(F,\beta)/\mathfrak{m}_{g,n}}
$$
along the projection morphism $\overline{\mathcal{M}}_U\,\longrightarrow\, \overline{\mathcal{M}}_{g,n}(F,\beta)$.
\end{itemize}
We know that $\phi_F$ is a (relative) perfect obstruction theory. Since $\phi_U$ is its base change along
the flat morphism, it follows that $\phi_U$ is also a perfect obstruction theory.

The cone of the composition of maps
\begin{align}\label{absPOT}
K^\bullet[-1]\, \xrightarrow{\,\,\,\phi[-1]\,\,\,}\, \mathbb{L}_{\overline{\mathcal{M}}(E/B)/{\mathfrak{m}_{g,n}
\times B}}[-1]\,\longrightarrow\, \mathbb{L}_{\mathfrak{m}_{g,n} \times B}\big\vert_{\overline{\mathcal{M}}(E/B)}
\end{align}
defines an absolute perfect obstruction theory because $\mathfrak{m}_{g,n} \times B$ is smooth;
we denote it by
\begin{equation}\label{a1}
\mathbb{E}_{\overline{\mathcal{M}}_{g,n}(E/B)}.
\end{equation}
The virtual dimension is 
$$
\text{fiberwise rank of } K^\bullet + \dim(\mathfrak{m}_{g,n} \times B),
$$
which is 
$$
\dim F\cdot (1-g)+\int_\beta c_1(T_F)+3(g-1)+n+\dim B,
$$ 
by the Riemann-Roch Theorem.
\end{proof}

The perfect obstruction theory $\mathbb{E}_{\overline{\mathcal{M}}_{g,n}(E/B)}$
in \eqref{a1} gives rise to the virtual cycle
$$
\left[\overline{\mathcal{M}}_{g,n}(E/B,\beta)\right]^{\mathrm{vir}}\,=\,
[\overline{\mathcal{M}}^{\mathrm{Fib}}_{g,n}(E/B,\widetilde{\beta})]^{\mathrm{vir}}
\,\in\, A_{\mathrm{vd}}(\overline{\mathcal{M}}^{\mathrm{Fib}}_{g,n}(E/B,\widetilde{\beta}))\otimes_{\mathbb{Z}}\mathbb{Q}.
$$

\subsection{Gromov-Witten invariants}

Throughout, we will assume $F$ is a non-singular projective variety, and $E/B$ is a $F$-fiber bundle with 
non-singular base $B$ and structure group $G$ as defined in \eqref{the structure group_1}. We have the 
Poincar\'{e} duality isomorphism
$$H^{2i}(E,\,\mathbb{Q}) \,\xrightarrow{\,\,\simeq\,\,}\, H_{2\dim E-2i}(E,\,\mathbb{Q}),\, \ \, 
\alpha \,\longmapsto\, \alpha \cap [E]. $$ Given $\beta$ an effective $1$-cycle on $F$, let 
$\widetilde{\beta}$ be the global section of $R_2\pi_* \underline{\mathbb{Z}}$ (cf. \eqref{R_2}) corresponding 
to $\beta$. When $B\,=\,\mathrm{Spec}(\mathbb{C})$, given $n$ cycles $\gamma_1,\, \ldots,\, \gamma_n
\,\in\, H^*(E)$, the 
ordinary Gromov-Witten invariant is defined by integrating the product $\prod_{i=1}^n 
\ev_i^*(\gamma_i)$ over the virtual class $[\overline{\mathcal{M}}_{g,n}(E/B,\beta)]^{\mathrm{vir}} 
\,=\,[\overline{\mathcal{M}}_{g,n}(F,\beta)]^{\mathrm{vir}}$.
The following definition extends the above.

\begin{defn}\label{GW definition}
Given cohomology classes $\gamma_1,\, \ldots,\, \gamma_n\,\in\, H^*(E,\,\mathbb{Q})$, define the \textit{fiberwise Gromov-Witten} \textnormal{(fiberwise GW for short)} invariant
$$
\langle \gamma_1, \ldots, \gamma_n \rangle^{\mathrm{Fib}}_{g,n, \beta}\ :=\
\int_{[\overline{\mathcal{M}}_{g,n}(E/B,\beta)]^{\mathrm{vir}}} \prod_{i=1}^n \ev_i^*(\gamma_i).
$$
\end{defn}

Note that $\langle \gamma_1, \ldots, \gamma_n \rangle^{\mathrm{Fib}}_{g,n, \beta}$ vanishes unless
$\sum\limits_ {i=1}^n \text{codim}(\gamma_i)\,=\, \mathrm{vd}$, where $\mathrm{vd}$ is given in \eqref{vdim_1}.

\subsection{Axioms of GW-invariants}

We now restate a few relations between fiberwise GW invariants, which are known to be true 
for the ordinary GW invariants. Their proofs are modifications of the proof in the ordinary case (the 
proofs in the ordinary case may be found in \cite{Cox.Katz, FuPa,K.M}).

First, consider the string axiom.
Roughly speaking, the string axiom asserts that if there is no condition in 
any of the marked point, the ordinary GW invariants vanish. In our set-up, we have the following:

\begin{property}[{String axiom}]\label{sa}
Assume that either $n \,>\,3$ or $\beta \,\neq\, 0$. Let $\gamma_1,\, \cdots, \,\gamma_n$ be cohomology classes
of $E$ such that $\gamma_n\,=\, \pi^*(\delta_n)$ for some class $\delta_n$ of $B$.
Then the fiberwise GW invariant $\langle \gamma_1,\, \ldots,\, \gamma_n \rangle^{\mathrm{Fib}}_{g,n, \beta}$ vanishes.
\end{property}

\begin{proof}
Since the evaluation maps in \eqref{evaluation map} are maps between fiber bundles, the composition $\widetilde{\pi}$ of 
the evaluation map with the projection map $\pi\,:\,E\,\longrightarrow\, B$ is the natural projection map 
$$\widetilde{\pi}\,:\,\overline{\mathcal{M}}_{g,n}(E/B,\beta)\,\longrightarrow\, B.$$ Using this, the cycle 
$\ev_n^*(\gamma_n)$ can be written as $\widetilde{\pi}^*(\delta_n)$. Since the projection $\widetilde{\pi}$ 
factors through $\overline{\mathcal{M}}_{g,n-1}(E/B,\beta)$, the product $\ev_1^*(\gamma_1)\cdot \cdots 
\ev_n^*(\gamma_n)$ on $\overline{\mathcal{M}}_{g,n}(E/B,\beta)$ is the pullback of $\ev_1^*(\gamma_1)\cdot 
\cdots \ev_{n-1}^*(\gamma_{n-1}) \cdot \widetilde{\pi}^*(\delta_n)$ on 
$\overline{\mathcal{M}}_{g,n-1}(E/B,\beta)$ via the map $$\pi_n \,:\, 
\overline{\mathcal{M}}_{g,n}(E/B,\beta)\,\longrightarrow\, \overline{\mathcal{M}}_{g,n-1}(E/B,\beta)$$ that 
forgets the last marked point (here the notation $\widetilde{\pi}$ is abused). Therefore, once we have 
\begin{equation}\label{a2}
(\pi_n)_*[\overline{\mathcal{M}}_{g,n}(E/B,\beta)]^{\mathrm{vir}}\,=\,0,
\end{equation}
we get the vanishing 
$$\int_{[\overline{\mathcal{M}}_{g,n}(E/B,\beta)]^{\mathrm{vir}}} \prod\limits_{i=1}^n \ev_i^*(\gamma_i)\,=\,0,$$
by the projection formula.

One way of proving \eqref{a2} is to use Manolache's virtual pushforward theorem \cite{Manolache}. Denoting
by $\mathbb{E}_n$ and $\mathbb{E}_{n-1}$ the perfect obstruction theory of
$\overline{\mathcal{M}}_{g,n}(E/B,\beta)$ and $\overline{\mathcal{M}}_{g,n-1}(E/B,\beta)$ respectively, we have an exact sequence
$$
\pi_n^*\mathbb{E}_{n-1}\,\longrightarrow\, \mathbb{E}_n\,\longrightarrow\, (R(\pi_n)_*\Omega_{\pi_n}^\vee)^\vee.
$$
Here we use the fact that $\pi_n \,:\, \overline{\mathcal{M}}_{g,n}(E/B,\beta)
\,\longrightarrow\, \overline{\mathcal{M}}_{g,n-1}(E/B,\beta)$ is the universal curve over $\overline{\mathcal{M}}_{g,n-1}(E/B,\beta)$. Since $\pi_n$ is smooth except the locus of the relative dimension $0$ over $\overline{\mathcal{M}}_{g,n-1}(E/B,\beta)$, the complex $(R(\pi_n)_*\Omega_{\pi_n}^\vee)^\vee$ is a perfect obstruction theory of $\pi_n$ so that we can apply Manolache's virtual pushforward theorem.
\end{proof}

\begin{rem}
Note that if $B\,=\,\mathrm{Spec}(\mathbb{C})$, then the only cycle on $B$ is $1$, in
which case, we recover the ordinary string axiom.
\end{rem}

\begin{rem}
The above proof of Property \ref{sa} suggests that we can get rid of the map $\pi_n$ which has positive 
dimensional fibers only if we avoid pulling back the cycles coming from the base $B$ through the evaluation 
maps. However it is natural to consider the evaluation of cycles involving ones coming from $B$. We can do that 
and avoid the complexity occurring in $\pi_n$ by directly pulling back cycles of $B$ through the bundle map 
$\widetilde{\pi}\,:\, \overline{\mathcal{M}}_{g,n}(E/B,\beta) \,\longrightarrow\, B$. This suggests the following 
modification of \Cref{GW definition}:
  
Given cohomology classes $\gamma_1,\, \ldots,\, \gamma_n\,\in\, H^*(E,\,\mathbb{Q})$ and
$\delta \,\in\, H^*(B,\, \mathbb{Q})$, we define the fiberwise Gromov-Witten invariant
$$
\langle \delta,\, \gamma_1,\, \ldots,\, \gamma_n \rangle^{\mathrm{Fib}}_{g,n, \beta}
\,:=\, \int_{[\overline{\mathcal{M}}_{g,n}(E/B,\beta)]^{\mathrm{vir}}} \widetilde{\pi}^* (\delta) \prod_{i=1}^n \ev_i^*(\gamma_i).
$$ 
But notice that $\langle \delta,\, \gamma_1,\, \ldots,\, \gamma_n \rangle^{\mathrm{Fib}}_{g,n, \beta}
\,=\, \langle  \pi^*(\delta)\cdot \gamma_1, \ldots, \gamma_n\rangle^{\mathrm{Fib}}_{g, n,\beta}.$ In fact, we
can attach $\pi^*\delta$ at any marked point and the value of the integral will be the same. 
\end{rem}

Now consider the divisor axiom. In the ordinary case, it states that any GW invariant that has the 
class of a divisor at one of the marked points is determined by the cycles at other marked points.

\begin{property}[{Divisor axiom}]\label{da}
Assume that either $n \,>\,3$ or $\beta \,\neq\, 0$. Let $D$ be a divisor on $F$, and let
$\gamma_1,\, \cdots,\, \gamma_{n+1}$ be cycles of $E$ such that the pullback of $\gamma_{n+1}$ to any
fiber is $D$ (via some isomorphism $F\,\cong\, E_b$ in the given class; recall
that the structure group of the fiber bundle $\pi$ is $G$). Then given a cycle $\delta$ on $B$,
$$
\langle \delta, \gamma_1, \ldots, \gamma_{n+1} \rangle^{\mathrm{Fib}}_{g,n+1, \beta}
\ =\ \left(\int_{\beta}D \right) \cdot \langle \delta, \gamma_1, \ldots, \gamma_n \rangle^{\mathrm{Fib}}_{g,n, \beta}.
$$
\end{property}

\begin{proof}
Using the decomposition of the cycle $$[\overline{\mathcal{M}}_{g,n}(E/B, \beta)]^{\mathrm{vir}}\,=\,
\sum_i\alpha_i$$ into irreducible components, Manolache's virtual pushforward theorem tells us that
$$
(\pi_{n+1})_*\left(\ev_{n+1}^*(\gamma_{n+1})\cap [\overline{\mathcal{M}}_{g,n+1}(E/B,\beta)]^{\mathrm{vir}}\right)
\, =\, \sum_i c_i \cdot \alpha_i 
$$
for some constants $c_i\,\in\,\mathbb{Q}$. Using the Gysin pullback to the inclusion map
of any (smooth) point $\{b\}\, \hookrightarrow\, B$, it recovers the usual divisor axiom, which implies
that $c_i\,=\, \int_\beta D$ for all $i$.
\end{proof}

Next, we move towards the computation of GW invariants when $\beta\,=\,0$. In this case, the moduli space
$\overline{\mathcal{M}}_{g,n}(E/B,\beta)$ is isomorphic to $\overline{\mathcal{M}}_{g,n} \times E$. The
genus zero case is straightforward. Notice that when $g \,=\, 0$, we have
 \[ [\overline{\mathcal{M}}_{0,n}(E/B,0)]^{\mathrm{vir}}\ =\ [\overline{\mathcal{M}}_{0,n}\times E].  \]
Therefore, the point mapping axiom in this case reduces to the following:
given cycles $\gamma_1,\, \ldots,\, \gamma_{n}$ on $E$, we have
\[
\langle  \gamma_1, \ldots, \gamma_n \rangle^{\mathrm{Fib}}_{0,n,0}\,=\,
\int_{\overline{\mathcal{M}}_{0,n}}1 \cdot \int_E \gamma_1 \cdots \gamma_n\,=\,
\begin{cases}
\int_E \gamma_1 \cdot \gamma_2 \cdot \gamma_3, & \text{ if } n=3,\\
0, & \text{ otherwise.}
\end{cases}\]

When $g \,\geq\, 1$, the obstruction sheaf $R^1\pi_*f^*T_{E/B}$ is isomorphic to $R^1\pi_* \mathcal{O}
\boxtimes T_{E/B}$. The latter is isomorphic to the vector bundle $\mathcal{H}^{\vee}\boxtimes T_{E/B}$,
of rank $e\,:=\, g \cdot \dim F$, where $\mathcal{H}$ denotes the Hodge bundle $\pi_* \omega$. Consequently, 
$$[\overline{\mathcal{M}}_{g,n}(E/B,0)]^{\mathrm{vir}}\ =\
c_e(\mathcal{H}^{\vee}\boxtimes T_{E/B})\cdot [\overline{\mathcal{M}}_{g,n}\times E].$$
Next, we will compute degree zero GW invariants for the case when $g \,=\, 1$. Higher genera cases can be computed similarly.

\begin{property}[{Point mapping axiom}] \label{pma}
Given cycles $\gamma_1,\, \ldots,\, \gamma_{n}$ of $E$, we have
\[
\langle  \gamma_1, \ldots, \gamma_n \rangle^{\mathrm{Fib}}_{1,n,0}\,=\,\,
\begin{cases}
-\dfrac{1}{24}\int_E c_{\dim F -1}(T_{E/B}) \cdot \gamma_1, & \text{ if } g\,=\,1 \text{ and }n\,=\,1,\\
~~ 0, & \text{ otherwise.}
\end{cases}
\]
\end{property}

\begin{proof}
When $g\,=\,1$, we have 
$$c_e(\mathcal{H}^{\vee}\boxtimes T_{E/B})\ =\ 1\otimes f^*c_e(T_{E/B})-\lambda_1 \otimes f^*c_{e-1}(T_{E/B}),$$
where $\lambda_1\,=\,e(\mathcal{H})$. 
Note that $\int_{\overline{\mathcal{M}}_{1,n}} 1 \,=\, 0$ for all $n \,\geq\, 1$. Therefore,
$$\langle  \gamma_1,\, \ldots,\, \gamma_n \rangle^{\mathrm{Fib}}_{1,n,0}
\,=\, -\int_{\overline{\mathcal{M}}_{1,n}}\lambda_1 \cdot \int_E c_{e-1}(T_{E/B}) \cdot \gamma_1 \cdots \gamma_n.$$
Similarly, due to the dimensional reason $\int_{\overline{\mathcal{M}}_{1,n}}\lambda_1$ is non-zero only
if $n\,=\,1$, and in that case the value of the integral is $1/24$, as follows from the proof of
\cite[Proposition 2.7]{Getzler}.
\end{proof}

Now we discuss a boundary divisor of $\overline{\mathcal{M}}_{g,n}(E/B,\beta)$. In the ordinary case, the
boundary divisor of $\overline{\mathcal{M}}_{g,n}(F, \beta)$ decomposes as
$$
D(g_1,P,\beta_1; g_2,Q, \beta_2)\ :=\ \overline{\mathcal{M}}_{g_1, P \cup \{ \bullet\}}(F, \beta_1)
\times_F \overline{\mathcal{M}}_{g_2, Q \cup \{ \bullet\}}(F, \beta_2),
$$ 
where $g_1+g_2\,=\,g$, $P \sqcup Q$ is a partition of $\{1,\, \ldots ,\,n\}$ and $\beta_1+\beta_2
\,=\,\beta$ (see \cite[Section 6]{FuPa} for more details). Similarly given $g_1,\,g_2$ with $g_1+g_2
\,=\,g$, a partition $\{ P,\, Q\}$ of $\{1,\, \ldots ,\,n\}$ and $\beta_1,\, \beta_2$ with $\beta_1+\beta_2
\,=\,\beta$, we define a divisor
\begin{equation}\label{divisor}
 D^{\mathrm{Fib}}(g_1,P,\beta_1; g_2,Q,\beta_2)\,:=\,
\overline{\mathcal{M}}_{g_1, P \cup \{ \bullet\}}(E/B, \beta_1) \times_{E}
\overline{\mathcal{M}}_{g_2, Q \cup \{ \bullet\}}(E/B, \beta_2)
\end{equation}
of $\overline{\mathcal{M}}_{g,n}(E/B,\beta)$. When it is nonzero, its virtual cycle is defined to be the
pullback of the virtual cycle $[\overline{\mathcal{M}}_{g,n}(E/B,\beta)]^{\mathrm{vir}}$
$$
[D^{\mathrm{Fib}}(g_1,P,\beta_1; g_2,Q,\beta_2)]^{\mathrm{vir}}\,:=\,
i^![\overline{\mathcal{M}}_{g,n}(E/B,\beta)]^{\mathrm{vir}},
$$
where $i\,:\, D^{\mathrm{Fib}}(g_1,P,\beta_1; g_2,Q,\beta_2)\,\hookrightarrow\,
\overline{\mathcal{M}}_{g,n}(E/B,\beta)$ is the embedding. The splitting axiom determines the evaluation of
cycles on the boundary divisors of this form.

To state the splitting property, we need the expression of the diagonal class of $E$
as a cycle. Let $\{T_0\,=\,1, T_1, \,\ldots ,\,T_m \}$ be a basis of $H^*(E,\mathbb{Q})$. Define
$$g_{ij}\:=\ \int_E T_i \cup T_j.$$Then the class of the diagonal $\Delta_E$ of $E$ inside $E \times E$ is given by 
$$[\Delta_E]\ =\ \sum_{i,j}g^{ij}T_i \otimes T_j \in H^*(E \times E, \mathbb{Q}),$$
where the matrix $\left(g^{ij}\right)$ is the inverse of $\left(g_{ij}\right)$.  

\begin{property}[{Splitting axiom}]\label{Sa}
Consider the divisor $D^{\mathrm{Fib}}(g_1,P,\beta_1; g_2,Q,\beta_2)$ defined as in \eqref{divisor}.
Given cycles $\gamma_1,\, \ldots,\, \gamma_{n}$ on $E$, the following holds:
\begin{align*}
&\int_{[D^{\mathrm{Fib}}(g_1,P,\beta_1; g_2,Q,\beta_2)]^{\mathrm{vir}}} \prod_{i=1}^n \ev_i^*(\gamma_i) \\
&\qquad\qquad\qquad = \,\, \sum_{i,j}g^{ij} \langle\{ \gamma_i: i \in P\}, T_i \rangle^{\mathrm{Fib}}_{g_1,|P|+1, \beta_1} \cdot \langle\{ \gamma_i: i \in Q\}, T_j\rangle^{\mathrm{Fib}}_{g_2, |Q|+1, \beta_2}.
\end{align*}
\end{property}

\begin{proof}
The idea of the proof is the same as in the ordinary case. For notational simplicity, we denote the moduli spaces 
$$
D^{\mathrm{Fib}}(g_1,P,\beta_1; g_2,Q,\beta_2),\ \overline{\mathcal{M}}_{g_1, P \cup \{ \bullet\}}(E/B, \beta_1),\ \overline{\mathcal{M}}_{g_2, Q \cup \{ \bullet\}}(E/B,\beta_2) \ \ \text{and}\ \ \overline{\mathcal{M}}_{g,n}(E/B,\beta)
$$ 
by $D$, $M_P$, $M_Q$ and $M$, respectively. Consider the fiber diagram
\begin{equation}\label{DIA}
\begin{gathered}
\xymatrix{
D\,=\,M_P \times_E M_Q \ar[rr]^{\iota} \ar[d]_{\mathrm{ev}} && M_P \times M_Q \ar[d]^{\mathrm{ev}^{\prime}} \\
E^n \times E  \ar[rr]_{d} && E^n \times E \times E.
}
\end{gathered}
\end{equation} 
Here $d$ is defined by the diagonal embedding of $E$ inside $E \times E$; the map $\mathrm{ev}$ is the
evaluations at the marked points and the node, and $\mathrm{ev}^{\prime}$ is given by all the $n+2$ evaluations. 

Since $d$ is regular, the Gysin pullback class $d^!([M_P]^{\mathrm{vir}}\times [M_Q]^{\mathrm{vir}})$, which is
a Chow cycle on $D$, is defined.

We will first prove the following equality: 
\begin{align}\label{D=MPMQ}
[D]^{\mathrm{vir}}\ =\ d^!\left([M_P]^{\mathrm{vir}}\times [M_Q]^{\mathrm{vir}}\right),
\end{align}
and explain how it is used in proving the splitting axiom later.

To prove the \eqref{D=MPMQ}, we study the perfect obstruction theories defining both sides. From the embeddings 
\begin{align} \label{embs}
i\,:\,D\,\hookrightarrow\, M\quad \text{ and }\quad \iota\,:\,D\,\hookrightarrow\, M_P \times M_Q,
\end{align}
we obtain exact triangles (in the derived category of $D$) of the perfect obstruction theories
\begin{align}\label{exactt1}
\mathbb{L}_{i}[-1] \,\longrightarrow\, \mathbb{E}_{M}\big\vert_D \,\longrightarrow\, \mathbb{E}_1\quad
\text{ and }\quad \mathbb{L}_{d}[-1]\big\vert_D \,\longrightarrow\, (\mathbb{E}_{M_P}
\boxplus \mathbb{E}_{M_Q})\big\vert_D\,\longrightarrow\, \mathbb{E}_2
\end{align}
which are lifts of the exact triangles of cotangent complexes
$$
\mathbb{L}_{i}[-1] \,\longrightarrow\, \mathbb{L}_{M}\big\vert_D\, \longrightarrow\, \mathbb{L}_D \quad
\text{ and }\quad \mathbb{L}_{\iota}[-1] \,\longrightarrow\, \mathbb{L}_{M_P\times M_Q}\big\vert_D
\,\longrightarrow\, \mathbb{L}_D.
$$ 
For now, we may understand both $\mathbb{E}_1$ and $\mathbb{E}_2$ as complexes sitting in the exact triangles 
in \eqref{exactt1}. We will give their explicit description soon. Then the virtual pullbacks of virtual 
fundamental classes (Kim-Kresch-Pantev (\cite{Kim-Kresch-Pantev}) and Manolache (\cite{Manolache})) applied to 
the embeddings \eqref{embs} and the triangles \eqref{exactt1} tell us that the virtual cycles defined by the 
perfect obstruction theories $\mathbb{E}_1\,\longrightarrow\, \mathbb{L}_D$ and $\mathbb{E}_2
\,\longrightarrow\,  \mathbb{L}_D$ are 
$i^![M]^{\mathrm{vir}}\,=\,[D]^{\mathrm{vir}}$ and $d^!([M_P]^{\mathrm{vir}}\times [M_Q]^{\mathrm{vir}})$ 
respectively. Hence \eqref{D=MPMQ} follows from $\mathbb{E}_1\,\cong\, \mathbb{E}_2$\footnote{In view 
of the work of Siebert (\cite[Theorem 4.6]{Siebert}) we don't have to show that this quasi-isomorphism commutes
with the morphisms to $\mathbb{L}_D$.}.

Our strategy is to show both $\mathbb{E}_1$ and $\mathbb{E}_2$ are quasi-isomorphic to the mapping cone
\begin{align}\label{COne}
\mathrm{Cone}\left((R\pi_*f^*T_{E/B})^\vee[-1]\to \mathbb{L}_{\mathfrak{m}_{0,P}\times\mathfrak{m}_{0,Q}}\big\vert_D\right)[-1]\ \longrightarrow\ \mathbb{L}_{B}\big\vert_D,
\end{align}
where $\pi\,:\,C\,\longrightarrow\,  D$ is the universal curve and $f\,:\,C\,\longrightarrow\, E$ is the
universal map.

We will prove that $\mathbb{E}_1\,\cong\,\mathrm{Cone}\eqref{COne}$. We claim that this
is implied by a quasi-isomorphism of complexes
\begin{align}\label{cone=cone}
\mathrm{Cone}\left(\mathbb{L}_{i}[-1]\to \mathbb{E}_{M/B}\big\vert_D\right)\,\cong\, \mathrm{Cone}\left((R\pi_*f^*T_{E/B})^\vee[-1]\to \mathbb{L}_{\mathfrak{m}_{0,P}\times\mathfrak{m}_{0,Q}}\big\vert_D\right).
\end{align}
To prove the claim, note that from \eqref{cone=cone} we obtain two exact triangles
$$
\mathbb{L}_B\big\vert_D \,\longrightarrow\, \mathbb{E}_1 \,\longrightarrow\,
\text{LHS of }\eqref{cone=cone} \quad\text{and}\quad \mathbb{L}_B\big\vert_D \,\longrightarrow\,
 \mathrm{Cone}\eqref{COne}\,\longrightarrow\, \text{RHS of }\eqref{cone=cone}
$$
as follows. The first one is obtained by the following diagram of exact triangles
$$
\xymatrix@C=8mm{
& \mathbb{L}_i[-1]\ar@{=}[r] \ar[d] & \mathbb{L}_i[-1] \ar[d]\\
\mathbb{L}_B\big\vert_D \ar@{=}[d]\ar[r] & \mathbb{E}_{M}\big\vert_D \ar[d] \ar[r] & \mathbb{E}_{M/B}\big\vert_D \ar[d] \\
\mathbb{L}_B\big\vert_D \ar[r] & \mathbb{E}_1 \ar[r]& \text{LHS of }\eqref{cone=cone},
}
$$
and the second one is obvious. Comparing the two triangles we
conclude that having $\mathbb{E}_1\,\cong\, \mathrm{Cone}\eqref{COne}$ is equivalent to having
the quasi-isomorphism in \eqref{cone=cone}. This proves the claim.

So we will prove \eqref{cone=cone}. It follows from the following facts:
\begin{itemize}
\item $\mathbb{L}_i\,\cong\, N^*_{\mathfrak{m}_{g_1,P}\times\mathfrak{m}_{g_2,Q}/ \mathfrak{m}_{g,n}}\big\vert_D[1]$,

\item $\mathbb{E}_{M/B}\big\vert_D\,\cong\, \mathrm{Cone}((R\pi_*f^*T_{E/B})^\vee[-1] \to \mathbb{L}_{\mathfrak{m}_{g,n}}\big\vert_D)$,
and

\item $\mathbb{L}_{\mathfrak{m}_{g,n}}\big\vert_D\,\longrightarrow\,
\mathbb{L}_{\mathfrak{m}_{g_1,P}\times \mathfrak{m}_{g_2,Q}}\big\vert_D\,\longrightarrow\,
N^*_{\mathfrak{m}_{g_1,P}\times\mathfrak{m}_{g_2,Q}/ \mathfrak{m}_{g,n}}\big\vert_D[1]$ is an exact triangle.
\end{itemize}
Combining all these, the quasi-isomorphism \eqref{cone=cone} is obtained at the bottom-right corner
of the following diagram of exact triangles:
$$
\xymatrix{
& \mathbb{L}_i[-1]\ar@{=}[r] \ar[d] & \mathbb{L}_i[-1] \ar[d]\\
(R\pi_*f^*T_{E/B})^\vee[-1] \ar@{=}[d]\ar[r] & \mathbb{L}_{\mathfrak{m}_{g,n}}\big\vert_D \ar[d] \ar[r] &
\mathbb{E}_{M/B}\big\vert_D \ar[d] \\
(R\pi_*f^*T_{E/B})^\vee[-1] \ar[r] & \mathbb{L}_{\mathfrak{m}_{g_1,P}\times \mathfrak{m}_{g_2,Q}}\big\vert_D
\ar[r]& \eqref{cone=cone}.
}
$$

Next we claim that $\mathbb{E}_2\,\cong\,\mathrm{Cone} \eqref{COne}$. Letting $(\mathbb{E}_{M_P}\boxplus_B
\mathbb{E}_{M_Q})\big\vert_D$ denote the cone of $\mathbb{L}_B\big\vert_D\,\longrightarrow\,
(\mathbb{E}_{M_P}\boxplus \mathbb{E}_{M_Q})\big\vert_D$, we
conclude that $\mathbb{E}_2\,\cong\,\mathrm{Cone} \eqref{COne}$ follows from the
following quasi-isomorphism of complexes
\begin{align}\label{Cone=Cone}
\mathrm{Cone}\left(\mathbb{L}_{d}[-1]\big\vert_D\to (\mathbb{E}_{M_P}\boxplus_B \mathbb{E}_{M_Q})\big\vert_D\right)
\,\cong\, \mathrm{Cone}\left((R\pi_*f^*T_{E/B})^\vee[-1]\to
\mathbb{L}_{\mathfrak{m}_{g_1,P}\times\mathfrak{m}_{g_2,Q}}\big\vert_D\right).
\end{align}
As before it follows from a comparison of the exact triangles
$$
\mathbb{L}_B\big\vert_D \,\longrightarrow\, \mathbb{E}_2 \longrightarrow \text{LHS of }\eqref{Cone=Cone} 
\quad\text{and}\quad \mathbb{L}_B\big\vert_D \,\longrightarrow\, \mathrm{Cone}\eqref{COne}
\,\longrightarrow\, \text{RHS of }\eqref{Cone=Cone}.
$$
The left one is obtained by the following diagram of exact triangles 
$$
\xymatrix@C=4.5mm{
& \mathbb{L}_d[-1]\big\vert_D\ar@{=}[r]\ar[d]  & \mathbb{L}_{d}[-1]\big\vert_D\ar[d] && \\  
\mathbb{L}_B\big\vert_D \ar@{=}[d]\ar[r]& (\mathbb{E}_{M_P}\boxplus\mathbb{E}_{M_Q})\big\vert_D \ar[d] \ar[r] & (\mathbb{E}_{M_P}\boxplus_B \mathbb{E}_{M_Q})\big\vert_D \ar[d] &&\\
\mathbb{L}_B\big\vert_D \ar[r] & \mathbb{E}_2 \ar[r]& \text{LHS of }\eqref{Cone=Cone}. 
}
$$
The right triangle is obvious. So we will prove that \eqref{Cone=Cone} holds.
It follows from the following facts:
\begin{itemize}
\item $\mathbb{L}_d[-1]\,\cong\, \mathbb{L}_E$,

\item $\mathbb{L}_{E/B}\big\vert_D\,\longrightarrow\, R\pi_*(f^*_1T_{E/B}\oplus f^*_2T_{E/B})^\vee
\,\longrightarrow\, (R\pi_*f^*T_{E/B})^\vee$ is an exact triangle, and

\item $\mathbb{L}_{B\times\mathfrak{m}_{g_1,P}\times\mathfrak{m}_{g_2,Q}}\big\vert_D\,\longrightarrow\,
(\mathbb{E}_{M_P}\boxplus_B \mathbb{E}_{M_Q})\big\vert_D
\,\longrightarrow\, R\pi_*(f^*_1T_{E/B}\oplus f^*_2T_{E/B})^\vee$ is also an exact triangle.
\end{itemize}
From these, the quasi-isomorphism \eqref{Cone=Cone} is obtained at the bottom-right corner of the following diagram of exact triangles
$$
\xymatrix{
\mathbb{L}_{E/B}[-1]\ar[r]\ar[d] & \mathbb{L}_{B}\big\vert_D\ar[r] \ar[d] & \mathbb{L}_d[-1]\big\vert_D \ar[d]\\
R\pi_*(f^*_1T_{E/B}\oplus f^*_2T_{E/B})^\vee[-1] \ar[d]\ar[r] &
\mathbb{L}_{B\times\mathfrak{m}_{g_1,P}\times\mathfrak{m}_{g_2,Q}}\big\vert_D \ar[d] \ar[r] &
(\mathbb{E}_{M_P}\boxplus_B \mathbb{E}_{M_Q})\big\vert_D \ar[d] \\
(R\pi_*f^*T_{E/B})^\vee[-1] \ar[r] & \mathbb{L}_{\mathfrak{m}_{g_1,P}\times
\mathfrak{m}_{g_2,Q}}\big\vert_D \ar[r]& \eqref{Cone=Cone}.
}
$$
Thus we conclude that $\mathbb{E}_1\,\cong\,\mathrm{Cone}\eqref{COne}
\,\cong\,\mathbb{E}_2$, which implies that \eqref{D=MPMQ} holds.

Now we claim that the splitting axiom follows from \eqref{D=MPMQ}. In view of the diagram \eqref{DIA},
the splitting axiom follows from
\begin{align}\label{haha}
\mathrm{ev}'_*\iota_*[D]^{\mathrm{vir}}&\,=\,
\mathrm{ev}'_*\left(\mathrm{ev}'^*[\Delta_E]\cap \left([M_P]^{\mathrm{vir}}\times [M_Q]^{\mathrm{vir}}\right)\right). 
\end{align}
The left-hand side equals $d_*\mathrm{ev}_*[D]^{\mathrm{vir}}$ by the functoriality of the proper pushforwards and the right-hand side equals 
\begin{align*}
[\Delta_E]\cap \mathrm{ev}'_*\left([M_P]^{\mathrm{vir}}\times [M_Q]^{\mathrm{vir}}\right)&\,=\,d_*d^!\mathrm{ev}'_*\left([M_P]^{\mathrm{vir}}\times [M_Q]^{\mathrm{vir}}\right)\\
&\,=\,d_*\mathrm{ev}_*d^!\left([M_P]^{\mathrm{vir}}\times [M_Q]^{\mathrm{vir}}\right)
\end{align*}
by the projection formula, the fact that $[\Delta_E]\,=\,d_*1\,=\,d_*d^!1$, the commutativity of the proper pushforwards and Gysin pullbacks. Hence \eqref{D=MPMQ} induces \eqref{haha}, and the claim is proved.
\end{proof}

\section{Quantum Cohomology}\label{quantum cohomology}

Throughout this section, we will work only with the $g\,=\,0$ situation. In the ordinary case, the quantum cohomology 
defines the ``quantum'' product on the cohomology with coefficients in the Novikov ring. The main point here is 
to show the associativity, which is equivalent to the WDVV equation \cite[Theorem 8.2.4]{Cox.Katz}. On the plane 
$\mathbb{P}^2$, the WDVV equation yields a recursion relation for counting rational curves in $\mathbb{P}^2$ 
satisfying insertion conditions. We now define the quantum product in the current set-up and find an analogous WDVV 
equation. In the next section, we will apply the WDVV equation to count rational curves in $\mathbb{P}^3$ that 
lie inside a hyperplane in $\mathbb{P}^3$. Such curves are called rational planar curves. We will discuss this in 
detail in the next section.

First fix a basis $\{T_0\,=\,1,\,T_1,\, \ldots,\, T_m\}$ of $H^*(E,\,\mathbb{Q})$. Put $\gamma\,:= \,
\sum_{i=1}^m t_i T_i$, where $t_i$ are formal variables. We can now define the $g\,=\,0$ Gromov-Witten
potential the same way as in the ordinary case; it is defined by 
\begin{equation}\label{GW potential}
\Phi(\gamma)\,\,:=\, \sum_{n \geq 3} \sum_{\beta} \dfrac{1}{n!}  \langle \gamma^n
\rangle^{\mathrm{Fib}}_{0,n, \beta},\quad\text{where}\quad \langle \gamma^n \rangle^{\mathrm{Fib}}_{0,n, \beta}
\,=\,\langle \gamma,\ldots,\gamma \rangle^{\mathrm{Fib}}_{0,n, \beta}.
\end{equation}
Using the arguments as in the proof of \cite[Lemma 15]{FuPa} for a homogeneous variety $F$ and fixed $n$,
we conclude that $\langle \gamma^n \rangle^{\mathrm{Fib}}_{0,n, \beta}$ is non-zero only for finitely many effective classes $\beta$. Thus $\Phi(\gamma)$ is a formal power series in $\mathbb{Q}[[t]]:= \mathbb{Q}[[t_0,\ldots,t_m]]$:
\begin{equation}\label{GW potential expanded}
\Phi\left( t_0,\ldots,t_m \right)\,=\, \sum_{n_0+\cdots+n_m \geq 3}\sum_{\beta} \langle T_0^{n_0}\cdots T_m^{n_m} \rangle^{\mathrm{Fib}}_{ 0,n,\beta}\, \, \dfrac{t_0^{n_0}}{n_0!} \cdots \dfrac{t_m^{n_m}}{n_m!}.
\end{equation} 
Here $\langle T_0^{n_0}\cdots T_m^{n_m} \rangle^{\mathrm{Fib}}_{ 0,n,\beta}\,=\,\langle T_0,\,\ldots,\, T_0,\, \ldots,
\, T_m,\ldots T_m \rangle^{\mathrm{Fib}}_{ 0,n,\beta}$.
Letting $\Phi_i\,:=\,\dfrac{\partial\Phi}{\partial t_i}$, and using \eqref{GW potential expanded}, we can show that 
$$
\Phi_{ijk}\ =\  \sum_{n \geq 0} \sum_{\beta}  \dfrac{1}{n!} \langle \gamma^n\cdot T_i \cdot T_j \cdot T_k\rangle^{\mathrm{Fib}}_{ 0,n+3,\beta},
$$
where $\langle \gamma^n \cdot T_i \cdot T_j \cdot T_k\rangle^{\mathrm{Fib}}_{0,n+3,\beta}
\,=\,\langle \gamma,\,\ldots,\,\gamma,\, T_i,\, T_j,\, T_k \rangle^{\mathrm{Fib}}_{0,n+3,\beta}.$
Now define the ``\textit{quantum}'' product as follows:
$$T_i * T_j\ :=\ \sum_{e,f} \Phi_{ije}g^{ef}T_f.$$

\begin{lmm}\label{le1}
The following holds: $T_0*T_i\,=\,T_i\,=\, T_i*T_0$ for all $0 \,\leq\, i \,\leq\, m$.
\end{lmm}

\begin{proof}
By the string axiom we have 
$$
\Phi_{0ij}\,=\, \langle T_0 ,\, T_i ,\, T_j \rangle^{\mathrm{Fib}}_{ 0,3,0}\,=\, \int_E T_i \cdot T_j\,=\,g_{ij}.
$$
Hence
$$T_i*T_0\,=\,\sum_{e,f} \Phi_{i0e}g^{ef}T_f\,=\, \sum_{e,f} g_{ie}g^{ef}T_f\,=\,T_i.$$
The other equality is symmetric.
\end{proof}

Lemma \ref{le1} shows that $T_0$ is a unit for the $*$ product. We now move towards the associativity of the $*$ 
product. Before proving it, we will first establish the WDVV equation in the fiber bundle 
set-up. The main idea to establish the equation is to consider the natural forgetful map 
$\overline{\mathcal{M}}^{\mathrm{Fib}}_{0,n+4}(E/B,\widetilde{\beta})\,\longrightarrow\, \overline{\mathcal{M}}_{0,4} 
$ that forgets the first $n$ marked points. Now observe that in $\overline{\mathcal{M}}_{0,4}\,\cong
\,\mathbb{P}^1$, the 
two divisors $D(1\, 2|3 \, 4)$ and $D(1\, 3|2 \, 4)$ (actually points in $\mathbb{P}^1$) are linearly equivalent. 
Pulling back we conclude that the following divisors are linearly equivalent:
$$
\sum_{P,Q}  \sum_{\beta_1+\beta_2=\beta} D^{\mathrm{Fib}}(0,P,\beta_1; 0,Q, \beta_2),\quad n+1,\, n+2 \,\in
\, P, \quad  n+3,\, n+4 \,\in\, Q
$$ 
where $P,\,Q$ is a partition of $[n+4]$, and
$$
\sum_{P,Q}  \sum_{\beta_1+\beta_2=\beta} D^{\mathrm{Fib}}(0,P,\beta_1; 0,Q,\beta_2),\quad n+1,\, n+3 \,\in\,
 P, \quad  n+2,\, n+4 \,\in\, Q.
$$ 
Now, the splitting axiom determines the evaluation of a cycle on such divisors. Evaluation of the cycle $\ev^*(\gamma^n)\cdot \ev_{n+1}^*(T_i) \cdot \ev_{n+2}^*(T_j) \cdot \ev_{n+3}^*(T_k) \cdot \ev_{n+4}^*(T_l)$ on each of the divisor yields the required WDVV equation. 

\begin{thm}[{WDVV Equation}]\label{WDVV theorem}
For $i,\,j,\,k,\,l$, the potential $\Phi$ satisfies the equation
\begin{equation*}
\sum_{e,f}\Phi_{ije}g^{ef}\Phi_{fkl}\,=\, \sum_{e,f}\Phi_{ike}g^{ef}\Phi_{fjl}.
\end{equation*}
\end{thm}

We omit the detailed proof. Instead we refer \cite[Theorem 8.2.4]{Cox.Katz}.

Let us now prove the associativity of $*$ as an application of WDVV equation.

\begin{thm}
For $1 \,\leq\, i,\,j,\,k \,\leq\, m$, $$T_i*(T_j * T_k)\ =\ (T_i *T_j)*T_k.$$
Thus $H^*(E,\,\mathbb{Q})\otimes_{\mathbb{Q}}\mathbb{Q}[[t]]$ is a commutative ring with identity $T_0$.
\end{thm}
\begin{proof}
This is straightforward. We have
$$(T_i *T_j)*T_k\,=\, \sum_{e,f}\sum_{p,q}\Phi_{ije}g^{ef}\Phi_{fkp}g^{pq}T_q,$$
and
$$T_i*(T_j * T_k)\,=\,\sum_{e,f}\sum_{p,q}\Phi_{jke}g^{ef}\Phi_{ifp}g^{pq}T_q.$$
Now the WDVV equation immediately gives us the equality. The ring structure is now a straightforward consequence.
\end{proof}

\begin{rem}
The WDVV equation is, in fact, equivalent to the associativity of $*$ product. Since for each $l$, there exists $q$ 
such that $g^{lq}\,\neq\, 0$, comparing the coefficients of $T_q$ in $(T_i *T_j)*T_k$ and $T_i*(T_j * T_k)$, the WDVV 
equation can be recovered.
\end{rem}

\section{Rational plane curves in a moving family of planes}\label{rational planar curve section}

\subsection{A new proof of Theorem \ref{thm_B}}

Recall that a curve in $\mathbb{P}^3$ is said to be planar when it lies inside a hyperplane in $\mathbb{P}^{3}$. We now 
describe the space of planar rational stable maps in $\mathbb{P}^3$ of a given degree $d\,\in\, \mathbb{Z}_{\geq 0}$. 
The hyperplanes in $\mathbb{P}^3$ are parametrized by the Grassmannian 
$\mathbb{G}\,:=\, G(3,4)$ of $3$-planes in $\mathbb{C}^4$. Then $E \,\xrightarrow{\,\,\,\pi\,\,\,}\,
\mathbb{G}$ --- the projectivization of the universal subbundle on $\mathbb{G}$ --- is the universal hyperplane. It is a 
$\mathbb{P}^2$-fiber bundle; it's structure group is $\mathrm{Aut}(\mathbb{P}^2)$. Note
that the group $G$ described in \eqref{the structure 
group_1} is $\mathrm{Aut}(\mathbb{P}^2)$ itself. Thus the class $d \,\in\, \mathbb{Z}\,\cong\, H_2(\mathbb{P}^2,
\, {\mathbb Z})$ has a corresponding family $\widetilde{d}\,
\in\, \Gamma(\mathbb{G},\, R_2\pi_*\underline{\mathbb{Z}})$, and we can consider the moduli space 
$\overline{\mathcal{M}}^{\mathrm{Fib}}_{0,n}(E/\mathbb{G},\widetilde{d})\,=\,\overline{\mathcal{M}}_{0,n}(E/\mathbb{G},d)$. 
This moduli space is the space of stable maps of planar rational curves of degree $d$ in $\mathbb{P}^3$.

\subsection*{Cohomology of $E$}

Our goal of this section is to answer Question \ref{genus-q-count-in-plane}. 
To do so, we will use the WDVV equation and deduce a recursion relation in Theorem \ref{thm_B}. 
First the cohomology ring of $E\,\xrightarrow{\,\,\,\pi\,\,\,}\, \mathbb{G}$ will be computed.

Using Leray-Hirsch theorem \cite[Proposition 9.10]{Eisenbud-Harris}, the cohomology ring is
\begin{equation*}
H^*(E, \,\mathbb{Q})\, =\, \dfrac{\mathbb{Q}[a,H]}{\left( a^4, H^3 - a H^2 + a^2 H -a^3 \right)},
\end{equation*}  
where $a$ is the hyperplane class in $\mathbb{G}$, and $H$ is the pullback of the hyperplane class
in $\mathbb{P}^3$ via the composition of maps $E\,\hookrightarrow\, \mathbb{G}\times\mathbb{P}^3
\,\longrightarrow\, \mathbb{P}^3$. Clearly,
\begin{equation}\label{x1}
\{a^iH^j\,\big\vert\,\, 0 \,\leq\, i \,\leq\, 3,\ 0 \,\leq\, j \,\leq\, 2\}
\end{equation}
is a cohomology basis of $H^*(E,\, \mathbb{Q})$. Denote $a^iH^j$ by $T_{ij}$, for convenience. 

One can compute the dual basis $\{T^{ij}\}$ with respect to the Poincar\'e pairing:
$$
T^{ij}\ =\ T_{3-i,2-j}-T_{4-i,1-j}+T_{5-i,-j}
$$
and setting $T_{ij}\,=\,0$ if $i\,>\,3$ or $j\,<\,0$.
So the diagonal class of $E$ is
\begin{equation}\label{diag}
[\Delta_E] \,=\, \sum_{i=0}^3 \sum_{j=0}^2 T_{ij} \otimes T_{(3-i)(2-j)} -
\sum_{i=1}^3 \sum_{j=0}^1 T_{ij} \otimes T_{(4-i)(1-j)}+ \left( T_{20}\otimes T_{30}\right)+
\left(T_{30}\otimes T_{20}\right). 
\end{equation}

\subsection*{GW potential}

Let $T_{03}\,:=\,T_{12}-T_{21}+T_{30}$ denote the cohomology class $H^3\,\in\, H^6(E)$. Instead of using the
basis in \eqref{x1}, we consider the set $$\{T_{03}\}\cup\{T_{ij}\,\big\vert\,\, 0 \,\leq\, i
\,\leq\, 3,\ 0 \,\leq\, j \,\leq\, 2\}$$ of linearly dependent elements and
the corresponding formal variables $t_{ij}$. Using this, {\em define} the new Gromov-Witten potential
function\footnote{Evaluating at $t_{03}\,=\,0$ recovers the usual one \eqref{GW potential}.}
\begin{align}
\label{Potential_primary}
\Phi\ =\ \sum_{d \geq 0} \sum_{\sum n_{ij} \geq 3}  \langle T_{00}^{n_{00}}, \ldots , T_{32}^{n_{32}} , T_{03}^{n_{03}}  \rangle^{\mathrm{Fib}}_{0,\sum n_{ij},d} ~~\dfrac{t_{00}^{n_{00}}}{n_{00}!} \cdots \dfrac{t_{32}^{n_{32}}}{n_{32}!}\cdot \dfrac{t_{03}^{n_{03}}}{n_{03}!}.
\end{align} 
The potential $\Phi$ can be written as $\Phi\,=\, \Phi_{d=0}+ \Phi_{d>0}$\footnote{The $d\,=\,0$ and $d\,>\,0$
parts are usually called the `classical' and `quantum' parts.}, namely the sum of $d\,=\,0$ and
$d\,>\,0$ parts. Using the string and point mapping axioms in Properties \ref{sa}, \ref{pma} respectively, we can write
\begin{align*}
\Phi_{d=0}\,=\,&
\sum_{\sum n_{ij} =3}  \left( \int_E T_{00}^{n_{00}}  \cdots  T_{32}^{n_{32}}T_{03}^{n_{03}} \right) \dfrac{t_{00}^{n_{00}}}{n_{00}!} \cdots \dfrac{t_{32}^{n_{32}}}{n_{32}!}\dfrac{t_{03}^{n_{03}}}{n_{03}!}, \quad \text{ and }\\
\Phi_{d>0}\,=\,&
\sum_{d > 0}\sum_{\substack{\sum n_{ij} \geq3\\(i, j) \neq (i, 0)}}  \langle T_{01}^{n_{01}}, \ldots , T_{32}^{n_{32}} , T_{03}^{n_{03}}  \rangle^{\mathrm{Fib}}_{0,\sum n_{ij},d} ~ \dfrac{t_{01}^{n_{01}}}{n_{01}!} \cdots \dfrac{t_{32}^{n_{32}}}{n_{32}!}\cdot \dfrac{t_{03}^{n_{03}}}{n_{03}!}.
\end{align*} 

\subsection*{WDVV equation}
Similar to how the WDVV equation in Theorem \ref{WDVV theorem} is derived from the splitting axiom in
Property \ref{Sa}, we obtain the new WDVV equation with the potential $\Phi$ in \eqref{Potential_primary}:
\begin{align}
\label{WDVV}
\sum_{\substack{0\leq i,k \leq 3 \\ 0 \leq j,l\leq 2}} \dfrac{\partial^3 \Phi}{\partial t_{01}
\partial t_{01}\partial t_{ij}}g^{(ij)(k l)} \dfrac{\partial^3 \Phi}{\partial t_{kl}\partial t_{02}\partial t_{02}}
\,=\, \sum_{\substack{0\leq i,k \leq 3 \\ 0 \leq j,l\leq 2}} \dfrac{\partial^3 \Phi}{\partial t_{01}\partial t_{02}\partial t_{ij}}g^{(ij)(k l)} \dfrac{\partial^3 \Phi}{\partial t_{k l}\partial t_{01}\partial t_{02}}.
\end{align}
We observe from the diagonal class \eqref{diag} that the terms in \eqref{WDVV} survive only when $(i,\,j)
\,=\,(3-k,\,2-l)$ or $(4-k,\,1-l)$ or $(5-k,\,-l)$. With $\Phi_{d>0}$, a lot of cancellations happen because it does not have $t_{i0}$ 
variables; it only survives when $i+k\,=\,3$ and $j\,=\,l\,=\,1$.

\subsection*{Strategy of the proof}

Our strategy to prove \Cref{thm_B} is to observe nontrivial relationships extracted from the explicit computation
of the WDVV equation in \eqref{WDVV}. For $\Phi_{d=0}$, we have differentiations
\begin{align*}
&\dfrac{\partial^3 \Phi_{d=0}}{\partial t_{01}\partial t_{01}\partial t_{ij}}=\begin{cases} 1 \text{ if } (i,j)=(3,0),(2,1) \\ 0 \text{ otherwise}\end{cases}\!\!\!,\ \ \dfrac{\partial^3 \Phi_{d=0}}{\partial t_{01}\partial t_{02}\partial t_{ij}}=\begin{cases} 1 \text{ if } (i,j)=(2,0) \\ 0 \text{ otherwise}\end{cases} \\
&\text{and }\ \dfrac{\partial^3 \Phi_{d=0}}{\partial t_{02}\partial t_{02}\partial t_{ij}}=0.
\end{align*}
Putting these computations into \eqref{WDVV}, and using $g^{(ij)(kl)}$ calculated in \eqref{diag}, the WDVV equation \eqref{WDVV} becomes
\begin{align}
\label{WDVV1}
&\dfrac{\partial^3 \Phi_{d>0}}{\partial t_{02}\partial t_{02}\partial t_{02}}
+\sum_{\substack{i+k=3 \\ j=l=1}} \dfrac{\partial^3 \Phi_{d>0}}{\partial t_{01}\partial t_{01}\partial t_{ij}} \dfrac{\partial^3 \Phi_{d>0}}{\partial t_{kl}\partial t_{02}\partial t_{02}}\\
&=2\left( \dfrac{\partial^3 \Phi_{d>0}}{\partial t_{01}\partial t_{02}\partial t_{12}} - \dfrac{\partial^3 \Phi_{d>0}}{\partial t_{01}\partial t_{02}\partial t_{21}}\right) + \sum_{\substack{i+k=3 \\ j=l=1}} \dfrac{\partial^3 \Phi_{d>0}}{\partial t_{01}\partial t_{02}\partial t_{ij}} \dfrac{\partial^3 \Phi_{d>0}}{\partial t_{k l}\partial t_{01}\partial t_{02}}. \nonumber
\end{align}
Then we prove \Cref{thm_B} by reading off the coefficient of $e^{dt_{01}}\frac{t^{r}_{02}}{r!}
\frac{t^{s}_{03}}{s!}u^{\theta}$ in the above equation \eqref{WDVV1} after evaluating $t_{ij}\,
\longmapsto\, u^it_{0j}$. Here $u$ is a formal variable.

\subsection*{Left-hand side of \eqref{WDVV1}}
The first term of the left-hand side is
\begin{align}\label{020202}
\dfrac{\partial^3 \Phi_{d>0}}{\partial t_{02}\partial t_{02}\partial t_{02}}=&
\sum_{d > 0} \sum_{\substack{\sum n_{ij} \geq 3\\(i, j) \neq (i, 0)}} \langle T_{01}^{n_{01}}, \ldots , T_{32}^{n_{32}} , T_{03}^{n_{03}}  \rangle^{\mathrm{Fib}}_{0,\sum n_{ij},d} ~ \dfrac{t_{01}^{n_{01}}}{n_{01}!} \cdots \dfrac{t_{02}^{n_{02}-3}}{(n_{02}-3)!}\cdots \dfrac{t_{32}^{n_{32}}}{n_{32}!}\cdot \dfrac{t_{03}^{n_{03}}}{n_{03}!}.
\end{align} 
Recall that the coefficients of the right-hand side of \eqref{020202} are fiberwise GW invariants
\begin{align*}
\langle a^{\theta_1}H^{m_1}, a^{\theta_2}H^{m_2}, \ldots ,a^{\theta_n}H^{m_n}  \rangle^{\mathrm{Fib}}_{0,n,d}\ =\  \int_{[\overline{\mathcal{M}}_{0,n}(E/\mathbb{G},d)]^{\mathrm{vir}}} \widetilde{\pi}^*(a^{\sum_i \theta_i})  \prod_{i=1}^{n} \ev_i^*(H^{m_i}).
\end{align*}
Applying the divisor axiom in Property \ref{da}, it becomes $d^kN_d(r,s,\sum_i \theta_i)$, where $k$ is the
number of $1$'s among $m_1,\, \cdots,\, m_n$, while $r$, $s$ are the number of $2$'s and $3$'s respectively. Putting
these into \eqref{020202}, and evaluating $t_{ij}\, \longmapsto, u^it_{0j}$, we obtain
\begin{equation}\label{222}
\dfrac{\partial^3 \Phi_{d>0}}{\partial t_{02}\partial t_{02}\partial t_{02}}\ =\ 
\sum_{d > 0}\sum_{r,s,\theta}  N_d(r,s,\theta) ~ e^{dt_{01}}  \dfrac{t_{02}^{r-3}}{(r-3)!} \dfrac{t_{03}^{s}}{s!} u^{\theta}.
\end{equation} 
Similarly, we have
\begin{align}\label{1122}
&\sum_{\substack{i+k=3 \\ j=l=1}} \dfrac{\partial^3 \Phi_{d>0}}{\partial t_{01}\partial t_{01}\partial t_{ij}} \dfrac{\partial^3 \Phi_{d>0}}{\partial t_{kl}\partial t_{02}\partial t_{02}}\\
&=
\sum_{d_1,d_2 > 0}\sum_{\substack{r_1,s_1,\theta_1 \\ r_2,s_2,\theta_2}}  d_1^3d_2 N_{d_1}(r_1,s_1,\theta_1) N_{d_2}(r_2,s_2,\theta_2) ~ e^{(d_1+d_2)t_{01}} \dfrac{t_{02}^{r_1+r_2-2}}{r_1!(r_2-2)!} \dfrac{t_{03}^{s_1+s_2}}{s_1!s_2!} u^{\theta_1+\theta_2-3}. \nonumber
\end{align} 

\subsection*{Right-hand side of \eqref{WDVV1}}
We keep doing the same computation to get
\begin{align}\label{1212-1221}
&\dfrac{\partial^3 \Phi_{d>0}}{\partial t_{01}\partial t_{02}\partial t_{12}} - \dfrac{\partial^3 \Phi_{d>0}}{\partial t_{01}\partial t_{02}\partial t_{21}}\\
&=
\sum_{d > 0}\sum_{r,s,\theta}  dN_d(r,s,\theta) ~ e^{dt_{01}}  \dfrac{t_{02}^{r-2}}{(r-2)!} \dfrac{t_{03}^{s}}{s!} u^{\theta-1}
-\sum_{d > 0}\sum_{r,s,\theta}  d^2N_d(r,s,\theta) ~ e^{dt_{01}}  \dfrac{t_{02}^{r-1}}{(r-1)!} \dfrac{t_{03}^{s}}{s!} u^{\theta-2}. \nonumber
\end{align} 
Also we have
\begin{align}\label{1212}
&\sum_{\substack{i+k=3 \\ j=l=1}} \dfrac{\partial^3 \Phi_{d>0}}{\partial t_{01}\partial t_{02}\partial t_{ij}} \dfrac{\partial^3 \Phi_{d>0}}{\partial t_{kl}\partial t_{01}\partial t_{02}}\\
&=
\sum_{d_1,d_2 > 0} \sum_{\substack{r_1,s_1,\theta_1 \\ r_2,s_2,\theta_2}} d_1^2d_2^2 N_{d_1}(r_1,s_1,\theta_1) N_{d_2}(r_2,s_2,\theta_2) ~ e^{(d_1+d_2)t_{01}} \dfrac{t_{02}^{r_1+r_2-2}}{(r_1-1)!(r_2-1)!} \dfrac{t_{03}^{s_1+s_2}}{s_1!s_2!} u^{\theta_1+\theta_2-3}. \nonumber
\end{align} 

\subsection*{Recursive formula} Now we would like to read off the coefficients of
$e^{dt_{01}}\frac{t_{02}^{r-3}}{(r-3)!}\frac{t_{03}^s}{s!}u^{\theta}$ from \eqref{222},
\eqref{1122}, \eqref{1212-1221}, \eqref{1212}. To do so, the condition $r \, \geq \, 3$ is necessary. Moreover, if $r \, \geq \, 3$, the vanishing of $N_1(r,s,\theta) $ follows from the dimensional reason. Thus we will compute the aforementioned coefficient when $r \, \geq \, 3$ and $d \ \geq \ 2$.  In \eqref{222} and \eqref{1212-1221}, they
are $N_d(r,s,\theta)$ and $dN_{d}(r-1,s,\theta+1)-d^2N_d(r-2,s,\theta+2)$ respectively. 
In \eqref{1122} and \eqref{1212}, we have
\begin{align*}
&\sum_{\substack{d_1,d_2 > 0\\ d_1+d_2=d}}\sum_{\substack{r_1+r_2=r-1, r_2\geq 2\\ s_1+s_2=s \\ \theta_1+\theta_2=\theta+3}}  {r-3 \choose r_1}{s \choose s_1} d_1^3d_2 N_{d_1}(r_1,s_1,\theta_1) N_{d_2}(r_2,s_2,\theta_2), \\
&\sum_{\substack{d_1,d_2 > 0\\ d_1+d_2=d}}\sum_{\substack{r_1+r_2=r-1, r_1,r_2\geq 1\\ s_1+s_2=s \\ \theta_1+\theta_2=\theta+3}} {r-3 \choose r_1-1}{s \choose s_1} d_1^2d_2^2 N_{d_1}(r_1,s_1,\theta_1) N_{d_2}(r_2,s_2,\theta_2),
\end{align*}
respectively. Putting all these into \eqref{WDVV1}, we obtain the following theorem:

\begin{thm}
Let $N_d(r,s,\theta)$ be the fiberwise GW invariant defined in \eqref{ndrst}.
Then it satisfies the following recursion relation for $d \,\geq\, 2$ and $r\,\geq\, 3$\footnote{Using
the equality $H^3\,=\,aH^2-a^2H+a^3$ and the divisor axiom in Property \ref{da}, we can check that the relation
is the same one as in \cite[Theorem 3.3]{BMS}.}:
\begin{align*}
&N_d(r,s,\theta)=  2dN_d(r-1,s,\theta+1)-2d^2N_d(r-2,s,\theta+2) \\
&+ \sum_{\substack{d_1,d_2 > 0\\ d_1+d_2=d}}\sum_{\substack{r_1+r_2=r-1\\ s_1+s_2=s \\ \theta_1+\theta_2=\theta+3}} \left( d_1^2d_2^2\binom{r-3}{r_1-1}- d_1^3d_2{r-3\choose r_1}\right) \binom{s}{s_1} N_{d_1}(r_1,s_1,\theta_1)N_{d_2}(r_2,s_2,\theta_2).\nonumber
\end{align*}
The convention here is: ${a \choose b}\,=\,0$ if $b\,<\,0$ or $b\,>\,a$.
\end{thm}

The initial values of the recursion are given by 
\begin{align*}
		N_d(r,s,\theta) & \,= \begin{cases}
		1 &  \mbox{if} ~~(d,r,s,\theta) =(1,0,2,1),(1,2,1,1), (1,1,1,2),(1,2,0,3),(2,2,3,0)\\ 
		0 & \mbox{\textnormal{otherwise with $d = 1$ and $r \geq 3$}}\\
		0 & \mbox{\textnormal{otherwise with $d\geq 1$ and $r\leq 2$}}.\end{cases} 
\end{align*}
For $1\ \leq \ d \ \leq \ 2$ and $r \ \leq \ 2$, the initial values are computed in \cite[Lemmas 3.1 and 3.2]{BMS}. The computation of the remaining initial values are straightforward. Indeed, for $r \ \geq \ 3$, $N_1(r,s,\theta)$ vanishes due to the dimensional reason. Also, using the equality $H^3\,=\,aH^2-a^2H+a^3$, we can check that
$N_d(r,s,\theta)\,=\,0$ if $s+\theta\,\geq\, 4$. In particular, we have $N_d(r,s,\theta)\,=\,0$ if
$d\,\geq\, 3$ and $r\,\leq\, 2$.

\subsection{Enumerative significance}\label{enumerative significance} As we have noted earlier, the fiberwise GW invariants for an arbitrary fiber bundle
need not capture the actual counting even if the fiber is a homogeneous projective variety.
However, in our case, the invariants $N_d(r,s,\theta)$ considered above are all actual countings. We would like to explain it. 

Denote by $\tau$ the composition of morphism $E\,\hookrightarrow\, \mathbb{G}\times \mathbb{P}^3
\,\xrightarrow{\,\,\mathrm{pr}_2\,\,}\, \mathbb{P}^3$.
The map $\tau$ induces a morphism $\widetilde{\tau}$ between the moduli spaces, and we get the following commutative diagram: 
\begin{equation*}
\begin{gathered}
\xymatrix{
\overline{\mathcal{M}}_{0,n}(E/\mathbb{G},d) \ar[rrd]_{\widetilde{\ev}_i}\ar[d]^{\ev_i}\ar[rr]^{\widetilde{\tau}}& & \overline{\mathcal{M}}_{0,n}(\mathbb{P}^3,d) \ar[d]^{\ev_i^{\prime}}\\
E\ar[rr]_{\tau}&& \mathbb{P}^3
}
\end{gathered}
\end{equation*} 
The key advantage of our case is that $\widetilde{\tau}$ is a locally closed immersion from a sub-locus
in which the stable maps are closed embeddings, and the target space $\mathbb{P}^3$ of the codomain
of $\widetilde{\tau}$ is a homogeneous variety. Note that $\mathbb{P}^3$ is isomorphic to
${\rm GL}_4(\mathbb{C})/P$, where $P$ is a parabolic subgroup of ${\rm GL}_4(\mathbb{C})$
that preserves the line $(0,\,0,\, 0,\, {\mathbb C})\,\subset\,\mathbb{C}^4$.

In \cite[Lemma 14]{FuPa}, Fulton-Pandharipande proved that the intersection of pullback subvarieties of a 
homogeneous variety by the evaluation map in a correct codimension is of dimension $0$, by translating the 
subvarieties if necessary. This result and the map $\widetilde{\tau}$ together give us the following:

\begin{prp}\label{tranverse intersection proposition} \label{prop52}
Let $\Gamma_1,\, \cdots ,\, \Gamma_n$ be pure dimensional subvarieties of $\mathbb{P}^3$, while
$\gamma_1,\, \cdots ,\,\gamma_n$ are corresponding cycles in $H^*(\mathbb{P}^3,\, \mathbb{Q})$, such that 
$$
\sum_{i=1}^n \mathrm{codim}\;(\Gamma_i\subset \mathbb{P}^3)\ =\ 3d+2+n.
$$ 
Suppose $d\,\geq\, 2$. Then for general points $h_1,\, \cdots ,\,h_n$ of ${\rm GL}_4(\mathbb{C})$, the scheme
theoretic intersection 
\begin{equation*}
\widetilde{\ev}_1^{-1}(h_1 \Gamma_1) \cap \ldots \cap \widetilde{\ev}_n^{-1}(h_n \Gamma_n)
\end{equation*}
 is a finite number of reduced points in the automorphism free part of $\overline{\mathcal{M}}_{0,n}(E/\mathbb{G},d)$, and 
$$
\langle \tau^*(\gamma_1),\, \ldots,\, \tau^*(\gamma_{n})\rangle^{\mathrm{Fib}}_{ 0,n,d}
\ =\ \#\widetilde{\ev}_1^{-1}(h_1 \Gamma_1) \cap \ldots \cap \widetilde{\ev}_n^{-1}(h_n \Gamma_n).
$$
\end{prp}

\begin{proof}
Denote the moduli space of stable maps in $\overline{\mathcal{M}}_{0,n}(E/\mathbb{G},d) $ with non-singular domains by $\mathcal{M}_{0,n}(E/\mathbb{G},d) $, and denote the intersection by
$$
\mathcal{M}^{*}_{0,n}(E/\mathbb{G},d)\ :=\ \mathcal{M}_{0,n}(E/\mathbb{G},d) \cap \overline{\mathcal{M}}^{*}_{0,n}(E/\mathbb{G},d),  
$$ 
where the substack $\overline{\mathcal{M}}^{*}_{0,n}(E/\mathbb{G},d)$ of objects with trivial automorphisms
is defined as in \Cref{properties of coarse moduli space}. The arguments in the proof of \cite[Lemma 13]{FuPa} can
be used in proving that $\mathcal{M}^{*}_{0,n}(E/\mathbb{G},d)$ is a dense open subspace in $\overline{\mathcal{M}}_{0,n}(E/\mathbb{G},d)$.

The main idea of the proof is to use $\mathcal{M}^{*}_{0,n}(E/\mathbb{G},d)$ instead of $\mathcal{M}^{*}_{0,n}(\mathbb{P}^3,d)$ in \cite[Lemma 14]{FuPa}, which is the inverse image of the morphism $\widetilde{\tau}$. Using Kleiman's transversality theorem (cf. \cite[Theorem 1.7]{Eisenbud-Harris}), Fulton-Pandharipande \cite[Lemma 14]{FuPa} shows that the scheme theoretic intersection 
$$
(\ev^{\prime}_1)^{-1}(h_1 \Gamma_1) \cap \ldots \cap (\ev^{\prime}_n)^{-1}(h_n \Gamma_n)
$$
for general points $h_1,\, \cdots ,\, h_n$ of ${\rm GL}_4(\mathbb{C})$ is transverse and supported
on $\mathcal{M}^*_{0,n}(\mathbb{P}^3,d)$ with the codimension being $3d+2+n$. Since the inverse image
$\widetilde{\tau}^{-1}(\mathcal{M}^*_{0,n}(\mathbb{P}^3,d))$ of the smooth locus
$\mathcal{M}^*_{0,n}(\mathbb{P}^3,d)$ is $\mathcal{M}^{*}_{0,n}(E/\mathbb{G},d)$, which is again
smooth, the scheme theoretic intersection 
$$
\widetilde{\tau}^{-1}\circ (\ev^{\prime}_1)^{-1}(h_1 \Gamma_1) \cap \ldots \cap \widetilde{\tau}^{-1}\circ (\ev^{\prime}_n)^{-1}(h_n \Gamma_n)
$$
is transverse and supported on $\mathcal{M}^*_{0,n}(E/\mathbb{G},d)$. When $d\,\neq\, 0,\,1$,
then $\widetilde{\tau}$ on $\mathcal{M}^{*}_{0,n}(E/\mathbb{G},d)$ is a closed embedding. Hence the codimension $3d+2+n$ remains the same, which is the dimension of $\mathcal{M}^*_{0,n}(E/\mathbb{G},d)$. This proves the first part of the proposition.

To prove the remaining part, consider the following Cartesian diagram
\begin{equation*}
\begin{gathered}
\xymatrix{
\cap_{i=1}^n \widetilde{\ev}_i^{-1}(h_i \Gamma_i) \ar[rr]\ar[d] && M\times \prod_{i=1}^n h_i \Gamma_i\ar[d]\\
M \ar[rr]_{\iota} && M\times \left( \mathbb{P}^3 \right)^n
}
\end{gathered}
\end{equation*}
where $M$ denotes the moduli space $\overline{\mathcal{M}}_{0,n}(E/\mathbb{G},d)$ and $\iota$ is the graph
of the morphism $(\widetilde{\ev}_1,\, \cdots ,\,\widetilde{\ev}_n)$. Now, we have
\begin{align*}
\langle \tau^*(\gamma_1), \ldots, \tau^*(\gamma_{n})\rangle^{\mathrm{Fib}}_{ 0,n,\widetilde{d}}
&= \int_{[M]^{\mathrm{vir}}} \prod_{i=1}^n \widetilde{\ev}_i^* ([h_i\Gamma_i])\\
&= \int_{[M]} \prod_{i=1}^n \widetilde{\ev}_i^* ([h_i\Gamma_i])\\ 
&= \deg \left(\iota^*[ M\times \prod_{i=1}^n h_i \Gamma_i]\right)\\
& =\deg \left([\cap_{i=1}^n \widetilde{\ev}_i^{-1}(h_i \Gamma_i)]\right)\\
&= \#\widetilde{\ev}_1^{-1}(h_1 \Gamma_1) \cap \ldots \cap \widetilde{\ev}_n^{-1}(h_n \Gamma_n).
\end{align*}
Here, the first equality is definition of fiberwise GW invariant, while the second equality follows from
$[M]^{\mathrm{vir}}\,=\,[M]$ when $g\,=\,0$ and the fact that the fiber of $E$ is $\mathbb{P}^2$;
the third equality is by definition of the graph morphism $\iota$, while the fourth equality comes from the
transversality of the intersection, and the fifth equality is the definition of the degree of a zero-dimensional scheme.
\end{proof}

\begin{cor}
For $d \,\geq\, 2$, the number $N_d(r,s,\theta)$ is enumerative.
\end{cor}
\begin{proof}
Recall the definition, 
$$
N_d(r,s,\theta)\ =\ \int_{[\overline{\mathcal{M}}_{0,r+s}(E/\mathbb{G},d)]^{\mathrm{vir}}} \prod_{i=1}^{r} \ev_i^*(H^2) \cdot \prod_{i=r+1}^{r+s} \ev_i^*(H^3)\cdot \widetilde{\pi}^*(a^{\theta}).
$$ 
When $\theta\,=\,0$, Proposition \ref{prop52} shows that $N_d(r,s,0)$ is enumerative
because $H^2$ and $H^3$ are pullback classes from $\mathbb{P}^3$. When $1 \,\leq\, \theta \,\leq\, 3$, there
exists a subvariety $V$ of the base space $\mathbb{G}$ such that $[V]\,=\,a^{\theta}$. So by replacing
$\mathbb{G}$ by $V$, it reduces to the case where $\theta\,=\,0$ case. Hence $N_d(r,s,\theta)$ is enumerative.
\end{proof}
\begin{rem}
The virtual dimension of the moduli space $\overline{\mathcal{M}}_{0,0}(E/\mathbb{G},1)$ is $5$ whereas the space of lines in $\mathbb{P}^3$ is the same as $G(2,4)$, the Grassmannian of $2$-planes in $\mathbb{C}^4$, which is of $4$ dimensional. Therefore the fiberwise GW invariants $N_1(r,s,\theta)$ are not expected to be enumerative.
\end{rem}

\section*{Acknowledgement}
\noindent We thank Ritwik Mukherjee for a discussion about the initial condition of the recursive formula. N. Das and A. Paul would like to thank him for his advice on pursuing this direction.  J. Oh and A. Paul first met each other at a program ``Vortex Moduli'' at ICTS-Bengaluru (Code: ICTS/Vort2023/02) where we have started collaborating together. A. Paul thanks Martijn Kool for several fruitful discussions. We are very grateful to the referee for helpful comments.

N. Das is supported by the INSPIRE faculty fellowship (Ref No: IFA21-MA 161) funded by the DST, Govt. of 
India. J. Oh is supported by the New Faculty Startup Fund from Seoul National University and the National 
Research Foundation of Korea (NRF) grant funded by the Korean government (MSIT)(RS-2024-00339364). A. Paul 
acknowledges the support from the Department of Atomic Energy, Government of India, under project no. RTI4001.
I. Biswas is partially supported by a J. C. Bose Fellowship (JBR/2023/000003).

\section*{Data availability statement}
No data were used or generated.

\section*{Conflict of Interest}

None of the authors has any conflict of interest to report.

\bibliography{Planar} 
\bibliographystyle{siam}
\end{document}